\documentclass[11pt]{article}
\usepackage{hyperref}
\usepackage{times}  
\usepackage{mathpazo}
\usepackage{amssymb,amsmath,amsthm}
\usepackage{epsfig}
\usepackage{tikz}

\newtheorem{theorem}{Theorem}[section]
\newtheorem{proposition}[theorem]{Proposition}
\newtheorem{definition}[theorem]{Definition}

\newtheorem{claim}[theorem]{Claim}
\newtheorem{lemma}[theorem]{Lemma}

\newtheorem{corollary}[theorem]{Corollary}

\newtheorem{remark}[theorem]{Remark}

\newcommand{\qedsymb}{\hfill{\rule{2mm}{2mm}}}
\renewenvironment{proof}[1][]{\begin{trivlist}
\item[\hspace{\labelsep}{\bf\noindent Proof#1:\/}] }{\qedsymb\end{trivlist}}

\def\calE{{\cal E}}

\def\R{\mathbb{R}}

\def\C{\mathbb{C}}
\def\N{\mathbb{N}}

\newcommand\Prob[2]{{\Pr_{#1}\left[ {#2} \right]}}

\newcommand{\inrprd}[1]{\langle #1 \rangle}

\newcommand{\true}{\mathsf{True}}
\newcommand{\false}{\mathsf{False}}
\newcommand{\inver}{\mathsf{IN}}
\newcommand{\outver}{\mathsf{OUT}}

\newcommand{\NP}{\mathsf{NP}}

\newcommand{\ch}{\mathop{\mathrm{ch}}}
\newcommand{\chv}{\mathop{\mathrm{ch\text{-}s}}}

\renewcommand{\epsilon}{\varepsilon}

\newcommand{\linspan}{\mathop{\mathrm{span}}}
\newcommand{\Fset}{\mathbb{F}}         

\begin{document}

\title{{\bf On the Subspace Choosability in Graphs}\footnote{An extended abstract of this work appeared in Proc. of the European Conference on Combinatorics, Graph Theory and Applications (EuroComb), 2021~\cite{EurocombVersion}.}}
\author{
Dror Chawin\thanks{School of Computer Science, The Academic College of Tel Aviv-Yaffo, Tel Aviv 61083, Israel. Research supported by the Israel Science Foundation (grant No.~1218/20).}
\and
Ishay Haviv\footnotemark[2]
}

\date{}

\maketitle

\begin{abstract}
A graph $G$ is said to be $k$-subspace choosable over a field $\Fset$ if for every assignment of $k$-dimensional subspaces of some finite-dimensional vector space over $\Fset$ to the vertices of $G$, it is possible to choose for each vertex a nonzero vector from its subspace so that adjacent vertices receive orthogonal vectors over $\Fset$.
The {\em subspace choice number} of $G$ over $\Fset$ is the smallest integer $k$ for which $G$ is $k$-subspace choosable over $\Fset$.
This graph parameter, introduced by Haynes, Park, Schaeffer, Webster, and Mitchell (Electron. J. Comb.,~2010), is inspired by well-studied variants of the chromatic number of graphs, such as the (color) choice number and the orthogonality dimension.

We study the subspace choice number of graphs over various fields.
We first prove that the subspace choice number of every graph with average degree $d$ is at least $\Omega(\sqrt{d/\ln d})$ over any field.
We then focus on bipartite graphs and consider the problem of estimating, for a given integer $k$, the smallest integer $m$ for which the subspace choice number of the complete bipartite graph $K_{k,m}$ over a field $\Fset$ exceeds $k$.
We prove upper and lower bounds on this quantity as well as for several extensions of this problem.
Our results imply a substantial difference between the behavior of the choice number and that of the subspace choice number.
We also consider the computational aspect of the subspace choice number, and show that for every $k \geq 3$ it is $\NP$-hard to decide whether the subspace choice number of a given bipartite graph over $\Fset$ is at most $k$, provided that $\Fset$ is either the real field or any finite field.
\end{abstract}

\section{Introduction}

{\em Graph coloring} is the problem of minimizing the number of colors in a vertex coloring of a graph $G$ where adjacent vertices receive distinct colors.
This minimum is known as the {\em chromatic number} of $G$ and is denoted by $\chi(G)$.
Being one of the most popular topics in graph theory, the graph coloring problem was extended and generalized over the years in various ways.
One classical variant, initiated independently by Vizing in 1976~\cite{Vizing76} and by Erd\"{o}s, Rubin, and Taylor in 1979~\cite{ERT79}, is that of {\em choosability}, also known as {\em list coloring}, which deals with vertex colorings with some restrictions on the colors available to each vertex.
A graph $G=(V,E)$ is said to be {\em $k$-choosable} if for every assignment of a set $S_v$ of $k$ colors to each vertex $v \in V$, there exists a choice of colors $c_v \in S_v$ that form a proper coloring of $G$ (that is, $c_v \neq c_{v'}$ whenever $v$ and $v'$ are adjacent in $G$). The {\em choice number} of a graph $G$, denoted $\ch(G)$, is the smallest integer $k$ for which $G$ is $k$-choosable.
It is well known that the choice number $\ch(G)$ behaves quite differently from the standard chromatic number $\chi(G)$. In particular, it can be arbitrarily large even for bipartite graphs (see, e.g.,~\cite{ERT79}). The choice number of graphs enjoys an intensive study in graph theory involving combinatorial, algebraic, and probabilistic tools (see, e.g.,~\cite{AlonSurvey93}). The computational decision problem associated with the choice number is unlikely to be tractable, because it is known to be complete for the complexity class $\Pi_2$ of the second level of the polynomial-time hierarchy even for bipartite planar graphs~\cite{ERT79,Gutner96,GutnerT09}.

Another interesting variant of graph coloring, introduced by Lov{\'{a}}sz~\cite{Lovasz79} in the study of Shannon capacity of graphs, is that of {\em orthogonal representations}, where the vertices of the graph do not receive colors but vectors from some given vector space.
A $t$-dimensional orthogonal representation of a graph $G=(V,E)$ over $\R$ is an assignment of a nonzero vector $x_v \in \R^t$ to every vertex $v \in V$, such that $\langle x_v, x_{v'} \rangle = 0$ whenever $v$ and $v'$ are adjacent in $G$.\footnote{Orthogonal representations of graphs are sometimes defined in the literature as orthogonal representations of the complement, namely, the definition requires vectors associated with {\em non-adjacent} vertices to be orthogonal.}
The {\em orthogonality dimension} of a graph $G$ over $\R$ is the smallest integer $t$ for which there exists a $t$-dimensional orthogonal representation of $G$ over $\R$.
The orthogonality dimension parameter is closely related to several other well-studied graph parameters, and in particular, for every graph $G$ it is bounded from above by the chromatic number $\chi(G)$. The orthogonality dimension of graphs and its extensions to fields other than the reals have found a variety of applications in combinatorics, information theory, and theoretical computer science (see, e.g.,~\cite[Chapter~10]{LovaszBook} and~\cite{Haemers79}).
As for the computational aspect, the decision problem associated with the orthogonality dimension of graphs is known to be $\NP$-hard over every field~\cite{Peeters96} (see also~\cite{GolovnevH20}).

In 2010, Haynes, Park, Schaeffer, Webster, and Mitchell~\cite{HaynesPSWM10} introduced another variant of the chromatic number of graphs that captures both the choice number and the orthogonality dimension. In this setting, which we refer to as {\em subspace choosability}, each vertex of a graph $G$ is assigned a $k$-dimensional subspace of some finite-dimensional vector space, and the goal is to choose for each vertex a nonzero vector from its subspace so that adjacent vertices receive orthogonal vectors. The smallest integer $k$ for which such a choice is guaranteed to exist for all possible subspace assignments is called the {\em subspace choice number} of the graph $G$, formally defined as follows.

\begin{definition}
For a graph $G=(V,E)$ and a function $f:V \rightarrow \N$, $G$ is {\em $f$-subspace choosable} over a field $\Fset$ if for every integer $t$ and for every assignment of subspaces $W_v \subseteq \Fset^t$ with $\dim(W_v)=f(v)$ to the vertices $v \in V$ (which we refer to as an $f$-subspace assignment), there exists a choice of a nonzero vector $x_v \in W_v$ for each vertex $v \in V$, such that $\langle x_v , x_{v'} \rangle = 0$ whenever $v$ and $v'$ are adjacent in $G$. For an integer $k$, the graph $G$ is {\em $k$-subspace choosable} over $\Fset$ if it is $f$-subspace choosable over $\Fset$ for the constant function $f$ defined by $f(v)=k$.
The {\em subspace choice number} of $G$ over $\Fset$, denoted $\chv(G,\Fset)$, is the smallest $k$ for which $G$ is $k$-subspace choosable over $\Fset$.
\end{definition}
\noindent
Here and throughout the paper, we associate with the real field $\R$ and with every finite field $\Fset$ the inner product defined by $\langle x,y \rangle = \sum{x_i y_i}$, whereas for the complex field $\C$ we consider, as usual, the one defined by $\langle x,y \rangle = \sum{x_i \overline{y_i}}$.

The work~\cite{HaynesPSWM10} has initiated the study of the subspace choice number of graphs over the real and complex fields.
Among other things, it was shown there that a graph is $2$-subspace choosable over $\R$ if and only if it contains no cycles.
We note that this is in contrast to the characterization given in~\cite{ERT79} for the (chromatic) $2$-choosable graphs, which include additional graphs such as even cycles. This implies that the choice number and the subspace choice number do not coincide even on the $4$-cycle graph. Over the complex field $\C$, however, it was shown in~\cite{HaynesPSWM10} that a graph is $2$-subspace choosable if and only if it either contains no cycles or contains only one cycle and that cycle is even. This demonstrates the possible effect of the field on the subspace choice number.
It further follows from~\cite{HaynesPSWM10} that for every graph $G$ and every field $\Fset$, it holds that $\chv(G,\Fset) \leq \Delta(G)+1$ where $\Delta(G)$ stands for the maximum degree in $G$. In fact, a similar argument shows that $\Delta(G)$ can be replaced in this bound by the degeneracy of $G$ (i.e., the smallest integer $k$ for which every subgraph of $G$ contains a vertex of degree at most $k$).

\subsection{Our Contribution}

The current work studies the subspace choice number of graphs over various fields.
Our first result provides a lower bound on the subspace choice number of a general graph over any field in terms of its average degree.
\begin{theorem}\label{thm:degreeIntro}
There exists a constant $c>0$ such that for every graph $G$ with average degree $d>1$ and for every field $\Fset$,
\[\chv(G,\Fset) > c \cdot \sqrt{\frac{d}{\ln d}}.\]
\end{theorem}
\noindent
The proof of Theorem~\ref{thm:degreeIntro} is based on a probabilistic argument.
For certain graphs, we provide an improved lower bound on the subspace choice number, avoiding the logarithmic term (see Theorem~\ref{thm:projective}).
This improvement relies on an explicit construction of finite projective planes.

We note that a result of Saxton and Thomason~\cite{SaxtonT15}, improving on a result of Alon~\cite{AlonDegrees}, asserts that for every graph $G$ with average degree $d$, it holds that $\ch(G) \geq (1+o(1)) \cdot \log_2 d$, where the $o(1)$ term tends to $0$ when $d$ tends to infinity.
Erd\"{o}s et al.~\cite{ERT79} proved that the choice number of the complete bipartite graph $K_{n,n}$ satisfies $\ch(K_{n,n}) = (1+o(1)) \cdot \log_2 n$, hence the lower bound of~\cite{SaxtonT15} is tight on these graphs.
Theorem~\ref{thm:degreeIntro} thus shows a substantial difference between the behavior of the subspace choice number and that of the standard choice number in terms of the average degree.

For the complete graph $K_n$, it is easy to see that $\chv(K_n,\Fset) = n$ whenever $\Fset$ is a field over which no nonzero vector is self-orthogonal, such as $\R$ and $\C$.
For finite fields, however, we show that the subspace choice number of $K_n$ is strictly smaller than $n$ for every sufficiently large $n$. This in particular shows that the subspace choice number over finite fields can be smaller than the choice number.

\begin{theorem}\label{thm:n-sqrt{n}}
There exists a constant $c>0$ such that for every sufficiently large integer $n$ and for every finite field $\Fset$,
\[\chv(K_n,\Fset) \leq n - c \cdot \sqrt{n}.\]
\end{theorem}

We next put our focus on complete bipartite graphs.
For the color choosability problem, it was observed in~\cite{ERT79} that the graph $K_{k,m}$ is $k$-choosable for every $m < k^k$ whereas $\ch(K_{k,m})=k+1$ for every $m \geq k^k$.
Considering the subspace choice number of these graphs, for every field $\Fset$ it holds that  $\chv(K_{k,m},\Fset) \leq k+1$, because $K_{k,m}$ is $k$-degenerate.
We consider here the problem of identifying the values of $m$ for which this $k+1$ upper bound is tight.
Namely, for an integer $k$ and a field $\Fset$, let $m(k,\Fset)$ denote the smallest integer $m$ for which it holds that $\chv(K_{k,m},\Fset) = k+1$.
We provide the following lower bound.
\begin{theorem}\label{thm:Lower}
For every integer $k$ and for every field $\Fset$, \[ m(k,\Fset) > \sum_{i=1}^{k-1}{ \Big \lfloor \frac{k-1}{i} \Big \rfloor}.\]
In particular, for every field $\Fset$ it holds that $m(k,\Fset) = \Omega(k \cdot \log k)$.
\end{theorem}

We next provide a general approach for proving upper bounds on $m(k,\Fset)$.
The following theorem reduces this challenge to constructing families of vectors with certain linear independence constraints.
\begin{theorem}\label{thm:GeneralUpper}
If there exists a collection of $m = k \cdot (t-1) +1$ nonzero vectors in $\Fset^k$ satisfying that every $t$ of them span the entire space $\Fset^{k}$, then $m(k,\Fset) \leq m$.
\end{theorem}
\noindent
The above theorem allows us to derive upper bounds on $m(k,\Fset)$ for various fields $\Fset$.
\begin{corollary}\label{cor:Intro_m}
Let $k$ be an integer and let $\Fset$ be a field.
\begin{enumerate}
  \item\label{itm_cor:1} If $|\Fset| \geq k^2-k+1$ then $m(k,\Fset) \leq k^2-k+1$.
  \item\label{itm_cor:2} If $\Fset$ is a finite field of size $q \geq k$ then $m(k,\Fset) \leq k \cdot \frac{q^{k-1}-1}{q-1}+1$.
\end{enumerate}
\end{corollary}
\noindent
We remark that the first item of Corollary~\ref{cor:Intro_m} is obtained by applying Theorem~\ref{thm:GeneralUpper} with collections of vectors that form the columns of Vandermonde matrices.
It implies that $m(k,\Fset) = O(k^2)$ whenever the field $\Fset$ is infinite or sufficiently large as a function of $k$, leaving us with a nearly quadratic gap from the lower bound given in Theorem~\ref{thm:Lower}.
This again demonstrates a significant difference between the behavior of the choice number and that of the subspace choice number.

In fact, Theorems~\ref{thm:Lower} and~\ref{thm:GeneralUpper} are proved in a more general form with respect to {\em asymmetric} subspace assignments, where the left and right vertices of the complete bipartite graphs might be assigned subspaces of different dimensions.
For the precise generalized statements, see Theorems~\ref{thm:LowerAs} and~\ref{thm:GeneralUpper_as}.
We note that this is analogous to the asymmetric setting of color choosability that was recently studied by Alon, Cambie, and Kang~\cite{AlonAsymm21}.

We particularly consider the bipartite graph $K_{2,m}$ whose left side consists of only two vertices.
For an integer $n$, we say that $K_{2,m}$ is $(n;2,2)$-subspace choosable over a field $\Fset$ if it is $f$-subspace choosable over $\Fset$ for the function $f$ that assigns the integer $n$ to one vertex of the left side and the integer $2$ to each of the other vertices.
We consider the problem of determining, for a given integer $n$, the smallest $m$ for which $K_{2,m}$ is $(n;2,2)$-subspace choosable over a given field $\Fset$, and prove the following.

\begin{theorem}\label{thm:Intro_K_2,m}
For every integer $n \geq 1$ the following holds.
\begin{enumerate}
  \item\label{itm:K_2,n_1} The graph $K_{2,n-1}$ is $(n;2,2)$-subspace choosable over every field $\Fset$.
  \item\label{itm:K_2,n_2} The graph $K_{2,n}$ is $(n;2,2)$-subspace choosable over $\C$.
  \item\label{itm:K_2,n_3} The graph $K_{2,n}$ is $(n;2,2)$-subspace choosable over $\R$ if and only if $n$ is odd.
\end{enumerate}
\end{theorem}

We finally consider the computational aspect of the subspace choice number and prove the following hardness result.

\begin{theorem}\label{thm:hardness}
Let $k \geq 3$ be an integer and let $\Fset$ be either $\R$ or some finite field.
Then, the problem of deciding whether a given bipartite graph $G$ satisfies $\chv(G,\Fset) \leq k$ is $\NP$-hard.
\end{theorem}
\noindent
The proof of Theorem~\ref{thm:hardness} is inspired by the approach taken in a proof due to Rubin~\cite{ERT79} for the $\Pi_2$-hardness of the decision problem associated with the (color) choice number. His proof involves a delicate construction of several gadget graphs used to efficiently map an instance of the $\Pi_2$-variant of the satisfiability problem to an instance of the color choosability problem. These gadgets, however, do not fit the setting of subspace choosability. In fact, the characterization of $2$-subspace choosable graphs over the reals, given in~\cite{HaynesPSWM10}, implies that the instances produced by the reduction of~\cite{ERT79} are never subspace choosable over this field. To overcome this difficulty, we construct and analyze a different gadget graph that allows us, combined with ideas of Gutner and Tarsi~\cite{Gutner96,GutnerT09}, to obtain the $\NP$-hardness result stated in Theorem~\ref{thm:hardness}. Our analysis involves a characterization, stated below, of the $2$-subspace choosable graphs over finite fields, extending the characterizations given in~\cite{HaynesPSWM10} for the real and complex fields.
\begin{proposition}\label{prop:chrachterization}
For every finite field $\Fset$, a graph is $2$-subspace choosable over $\Fset$ if and only if it contains no cycles.
\end{proposition}
\noindent
While Theorem~\ref{thm:hardness} indicates the hardness of efficiently determining the subspace choice number of bipartite graphs, it would be natural to expect the stronger notion of $\Pi_2$-hardness to hold for this problem.

\subsection{Outline}
The rest of the paper is organized as follows.
In Section~\ref{sec:degree}, we prove Theorem~\ref{thm:degreeIntro}, relating the subspace choice number of a graph over a general field to its average degree. We also prove there an improved bound for certain graphs and discuss a limitation of our approach.
In section~\ref{sec:K_n}, we prove the upper bound on the subspace choice number of complete graphs over finite fields given in Theorem~\ref{thm:n-sqrt{n}}.
In Section~\ref{sec:2-vec}, we prove the characterization of $2$-subspace choosable graphs over finite fields given in Proposition~\ref{prop:chrachterization}, which will be used in the following sections.
In Section~\ref{sec:complete_bipartite}, we prove several upper and lower bounds on the subspace choosability of complete bipartite graphs in the asymmetric setting, and in particular confirm Theorems~\ref{thm:Lower},~\ref{thm:GeneralUpper}, and~\ref{thm:Intro_K_2,m}.
Finally, in Section~\ref{sec:hardness}, we prove our hardness result given in Theorem~\ref{thm:hardness}.

\section{Subspace Choosability and Average Degree}\label{sec:degree}

In this section we relate the subspace choice number of a graph over a general field to its average degree and prove Theorem~\ref{thm:degreeIntro}.
We start with the following definition of $k$-partitioned graphs (for an example, see Lemma~\ref{lemma:projective}).

\begin{definition}\label{def:partitioned}
Let $G=(V,E)$ be a graph.
For every vertex $v \in V$, let $E_v \subseteq E$ denote the set of edges of $G$ that are incident with $v$.
We say that the graph $G$ is {\em $k$-partitioned} if it is possible to partition every set $E_v$, $v \in V$, into $k$ sets $E_v^{(1)}, \ldots, E_v^{(k)}$ (some of which may be empty), such that for every function $g: V \rightarrow [k]$ there exist two adjacent vertices $v_1,v_2 \in V$ such that $\{v_1, v_2\} \in E_{v_1}^{(g(v_1))} \cap E_{v_2}^{(g(v_2))}$.
\end{definition}

The following theorem shows that the subspace choice number of a $k$-partitioned graph exceeds $k$ over any field.
\begin{theorem}\label{thm:k-partitioned}
For every $k$-partitioned graph $G$ and for every field $\Fset$, $\chv(G,\Fset) > k$.
\end{theorem}

\begin{proof}
Fix an arbitrary field $\Fset$. Let $G=(V,E)$ be a $k$-partitioned graph, and for every vertex $v \in V$, let $E_v = E_v^{(1)} \cup \cdots \cup E_v^{(k)}$ be the corresponding partition of the edges incident with $v$, as in Definition~\ref{def:partitioned}.
We use these partitions to define a $k$-subspace assignment over $\Fset$ to the vertices of $G$ involving vectors from the space $\Fset^{|E|}$, where each entry corresponds to an edge $e \in E$.
To a vertex $v \in V$ we assign the subspace $W_v$ spanned by the $k$ vectors $w^{(1)}_v, \ldots, w^{(k)}_v$, where $w^{(i)}_v$ is the $0,1$ indicator vector of the subset $E_v^{(i)}$ of $E$.
In fact, some of the sets $E^{(i)}_v$ might be empty, and thus some of the vectors $w^{(i)}_v$ might be zeros, resulting in subspaces $W_v$ of dimension smaller than $k$.
To fix it, one can increase the length of the vectors from $|E|$ to $|E|+k \cdot |V|$ and to add to each of the $k \cdot |V|$ vectors $w^{(i)}_v$ a nonzero entry in a coordinate on which all the others have zeros. These entries ensure that the dimension of every subspace $W_v$ is precisely $k$.
For simplicity of notation, we refer from now on to these modified vectors as $w^{(i)}_v$.

We show now that no choice of nonzero vectors from these subspaces satisfies that every two adjacent vertices receive orthogonal vectors over $\Fset$.
To see this, consider some choice of a nonzero vector $x_v \in W_v$ for each vertex $v \in V$.
We define a function $g: V \rightarrow [k]$ as follows.
For every $v \in V$, $x_v$ is a nonzero linear combination of the vectors $w^{(1)}_v, \ldots, w^{(k)}_v$, hence there exists some $j_v \in [k]$ for which the coefficient of $w^{(j_v)}_v$ in this linear combination is nonzero. We define $g(v)$ to be such an index $j_v$.
By assumption, there exist two adjacent vertices $v_1,v_2 \in V$ such that $\{v_1, v_2\} \in E_{v_1}^{(g(v_1))} \cap E_{v_2}^{(g(v_2))}$.
This implies that the entry that corresponds to the edge $\{v_1,v_2\}$ of $G$ is nonzero in both $x_{v_1}$ and $x_{v_2}$.
However, the supports of the subspaces $W_{v_1}$ and $W_{v_2}$ intersect at this single entry, implying that the vectors $x_{v_1}$ and $x_{v_2}$ are not orthogonal over $\Fset$.
This implies that there exists a $k$-subspace assignment over $\Fset$ to the vertices of $G$ with no appropriate choice of nonzero vectors, yielding that $\chv(G,\Fset) > k$, as required.
\end{proof}

Theorem~\ref{thm:k-partitioned} motivates the problem of determining the largest integer $k$ for which a given graph is $k$-partitioned.
The following lemma uses a probabilistic argument to prove a lower bound on this quantity in terms of the average degree.

\begin{lemma}\label{lemma:k-partitioned}
There exists a constant $c>0$ such that every graph with average degree $d>1$ is $k$-partitioned for some $k \geq c \cdot \sqrt{\frac{d}{\ln d}}$.
\end{lemma}

\begin{proof}
Let $G=(V,E)$ be a graph with average degree $d>1$. Note that $2 \cdot |E| = |V| \cdot d$.
Let $k$ be the largest integer satisfying
\begin{eqnarray}\label{eq:d_vs_k}
d > 2k^2 \cdot \ln k.
\end{eqnarray}
Observe that for an appropriate choice of the constant $c$, it holds that $k \geq c \cdot \sqrt{\frac{d}{\ln d}}$.
We prove that $G$ is $k$-partitioned by a probabilistic argument.
For every vertex $v \in V$, we define a random partition of the set $E_v$ of the edges incident with $v$, into $k$ sets $E_v^{(1)}, \ldots, E_v^{(k)}$ (some of which may be empty) as follows. For each edge $e \in E_v$, we pick at random, uniformly and independently, some $j \in [k]$, and put $e$ in $E_v^{(j)}$.
We claim that the obtained partitions satisfy with positive probability the condition given in Definition~\ref{def:partitioned}, namely, that for every function $g: V \rightarrow [k]$ there exist two adjacent vertices $v_1,v_2 \in V$ such that
\begin{eqnarray}\label{eq:v_1,v_2}
\{v_1, v_2\} \in E_{v_1}^{(g(v_1))} \cap E_{v_2}^{(g(v_2))}.
\end{eqnarray}
Indeed, for every fixed function $g: V \rightarrow [k]$ and for every edge $\{v_1,v_2\} \in E$, the probability that the event~\eqref{eq:v_1,v_2} occurs is $1/k^2$.
Hence, the probability that for all edges of $E$ this event does not occur is $(1-1/k^2)^{|E|}$. By the union bound, the probability that there exists a function $g: V \rightarrow [k]$ such that for all edges of $E$ the event~\eqref{eq:v_1,v_2} does not occur is at most
\[ k^{|V|} \cdot \bigg (1-\frac{1}{k^2} \bigg )^{|E|} \leq k^{|V|} \cdot e^{-|E|/k^2} = \big (e^{\ln k - d/(2 k^2)} \big )^{|V|}.\]
By~\eqref{eq:d_vs_k}, the above is smaller than $1$, hence with positive probability the random partition satisfies the required condition, and thus $G$ is $k$-partitioned
\end{proof}
\noindent
Combining Theorem~\ref{thm:k-partitioned} and Lemma~\ref{lemma:k-partitioned} completes the proof of Theorem~\ref{thm:degreeIntro}.

It is natural to ask whether Theorem~\ref{thm:k-partitioned} can be used to obtain better lower bounds on the subspace choice number of graphs than the one achieved by Theorem~\ref{thm:degreeIntro}.
The following lemma shows that for graphs with similar average and maximum degrees, Lemma~\ref{lemma:k-partitioned} is tight up to the logarithmic term.
Hence, the approach suggested by Theorem~\ref{thm:k-partitioned} cannot yield significantly better bounds for such graphs.
\begin{lemma}\label{lemma:not-k-partitioned}
There exists a constant $c>0$ such that every graph with maximum degree $D$ is not $k$-partitioned whenever $k \geq c \cdot \sqrt{D}$.
\end{lemma}

The proof of Lemma~\ref{lemma:not-k-partitioned} uses the Lov{\'a}sz local lemma stated below (see, e.g.,~\cite[Chapter~5]{AlonS16}).

\begin{lemma}[Lov{\'a}sz Local Lemma]\label{lemma:local}
Let $\calE$ be a collection of events such that for each $A \in \calE$, it holds that $\Prob{}{A} \leq p <1$ and that $A$ is mutually independent of a set of all but at most $d$ of the other events of $\calE$. If $e \cdot p \cdot (d+1) \leq 1$, then with positive probability none of the events of $\calE$ occurs.
\end{lemma}

\begin{proof}[ of Lemma~\ref{lemma:not-k-partitioned}]
Let $G=(V,E)$ be a graph with maximum degree $D$, and let $k \geq c \cdot \sqrt{D}$ be an integer for some constant $c$ to be determined.
We prove that $G$ is not $k$-partitioned by a probabilistic argument.
For every vertex $v \in V$, consider a partition $E_v = E_v^{(1)} \cup \cdots \cup E_v^{(k)}$ of the set of the edges incident with $v$ into $k$ sets (some of the sets may be empty). We claim that there exists a function $g: V \rightarrow [k]$ such that no edge $\{v_1, v_2\}$ of $G$ satisfies $\{v_1, v_2\} \in E_{v_1}^{(g(v_1))} \cap E_{v_2}^{(g(v_2))}$.
To prove it, consider a random function $g: V \rightarrow [k]$ such that each value $g(v)$ for $v \in V$ is chosen uniformly and independently at random from $[k]$.
For every edge $e = \{v_1, v_2\} \in E$, let $A_e$ denote the event that $e \in E_{v_1}^{(g(v_1))} \cap E_{v_2}^{(g(v_2))}$.
The probability of each event $A_e$ is clearly $1/k^2$.
In addition, every event $A_e$ is mutually independent of the set of all the other events $A_{e'}$ but those satisfying $e \cap e' \neq \emptyset$, whose number is at most $2 \cdot (D-1)$.
By the Lov{\'a}sz local lemma (Lemma~\ref{lemma:local}), it follows that if
\[e \cdot \tfrac{1}{k^2} \cdot (2D-1) \leq 1\]
then with positive probability no event $A_e$ occurs.
This implies that for an appropriate choice of the constant $c$, there exists a function $g$ with the required property. Since this holds for all possible partitions of the sets $E_v$ into $k$ sets, it follows that $G$ is not $k$-partitioned, and we are done.
\end{proof}

We end this section by proving that for certain graphs, the logarithmic term in Lemma~\ref{lemma:k-partitioned} can be avoided.
Here, the proof does not use a probabilistic construction of partitions, but an explicit one, based on finite projective planes.
\begin{lemma}\label{lemma:projective}
For a prime power $q$, let $H$ be the $(q+1)$-partite graph with $q$ vertices in every part.
Let $G$ be a graph obtained from $H$ by removing at most $q-1$ of its vertices.
Then, $G$ is $q$-partitioned.
\end{lemma}

\begin{proof}
The proof is based on a well-known construction of projective planes, some of whose properties are described next (see, e.g.,~\cite[Chapter~9]{Cameron94}).
For every prime power $q$, there exists a collection of $n = q^2+q+1$ elements called {\em points}, and $n$ sets of points, called {\em lines}, satisfying that every two lines intersect at a single point, every two points belong together to a single line, every point belongs to precisely $q+1$ of the lines, and every line includes precisely $q+1$ of the points.
Fix some point $p$, let $L_1, \ldots, L_{q+1}$ be the $q+1$ lines that include $p$, and put $L'_i = L_i \setminus \{p\}$ for every $i \in [q+1]$.
Note that the sets $L'_i$ are pairwise disjoint.
We view the graph $H$ as the graph on the vertex set $\cup_{i \in [q+1]}{L'_i}$ in which two vertices are adjacent if they belong to distinct sets $L'_i$.
Observe that every two vertices of $H$ are adjacent if and only if the line that includes their points does not include $p$.

Let $G=(V,E)$ be some subgraph of $H$ obtained by removing at most $q-1$ of its vertices, and observe that the number of its vertices satisfies
\[|V| \geq (q+1) \cdot q -(q-1) = q^2+1.\]
We show that $G$ is $q$-partitioned.
To do so, we assign to every edge of the graph $G$ the line that includes the points represented by its vertices.
This assignment induces for every vertex $v \in V$ a partition of the set $E_v$ of the edges incident with $v$ in $G$, where the sets of the partition correspond to the lines associated with the edges.
Observe that this partition of $E_v$ consists of at most $q$ sets.
Indeed, the vertex $v$ represents a point that belongs to $q+1$ lines, but no edge of $E_v$ is assigned the line that includes the point $p$ and the point of $v$.

In order to show that these partitions satisfy the condition of Definition~\ref{def:partitioned}, we shall verify that if one chooses for every point represented by a vertex in $G$ a line that corresponds to an edge incident with it, then there exist two adjacent vertices in $G$ for which the same line was chosen.
This indeed follows from the fact that no edge of $G$ corresponds to a line that includes $p$, hence the total number of lines associated with the edges of $G$ is at most $q^2$. Since the number of vertices in $G$ exceeds $q^2$, it follows that two vertices are assigned the same line.
Since this line does not include $p$, the two vertices must be adjacent in $G$, completing the proof.
\end{proof}

Note that the graph $H$ from Lemma~\ref{lemma:projective} is regular with degree $q^2$, hence the minimum degree of its subgraph $G$ is at least $q^2-q+1$, and yet $G$ is $q$-partitioned. This shows that the logarithmic term from Lemma~\ref{lemma:k-partitioned} is not needed for $G$.
By combining Lemma~\ref{lemma:projective} with Theorem~\ref{thm:k-partitioned}, we derive the following.
\begin{theorem}\label{thm:projective}
For a prime power $q$, let $H$ be the $(q+1)$-partite graph with $q$ vertices in every part.
Let $G$ be a graph obtained from $H$ by removing at most $q-1$ of its vertices.
Then, for every field $\Fset$, $\chv(G,\Fset) > q$.
\end{theorem}

\section{Subspace Choosability in Complete Graphs over Finite Fields}\label{sec:K_n}

In this section we prove Theorem~\ref{thm:n-sqrt{n}}, which provides an upper bound on the subspace choice number of complete graphs over finite fields.
We start with two useful lemmas.
\begin{lemma}\label{lemma:ortho_vec}
For a finite field $\Fset$ and an integer $t$, let $w_1,w_2,w_3$ and $z_1,z_2,z_3$ be two triples of vectors in $\Fset^t$.
Then, there exist $\alpha_1, \alpha_2, \alpha_3 \in \Fset$, not all zeros, such that
\[ \big \langle \sum_{i \in [3]}{\alpha_i \cdot w_i}, \sum_{i \in [3]}{\alpha_i \cdot z_i} \big \rangle = 0.\]
\end{lemma}

\begin{proof}
Consider the function $f: \Fset^3 \rightarrow \Fset$ defined by
$f(\alpha_1,\alpha_2,\alpha_3) = \big \langle \sum_{i \in [3]}{\alpha_i \cdot w_i}, \sum_{i \in [3]}{\alpha_i \cdot z_i} \big \rangle$.
The function $f$ is a degree $2$ polynomial on $3$ variables over $\Fset$, and $(0,0,0)$ forms a root of $f$.
The Chevalley theorem (see, e.g.,~\cite[Chapter~IV,~Theorem~1D]{Schmidt76}) implies that $f$ has another root, as required.
\end{proof}

\begin{lemma}\label{lemma:choice_3sub}
For a finite field $\Fset$ and an integer $t$, let $U_1, U_2, U_3$ be three subspaces of $\Fset^t$ whose dimensions satisfy
\[\dim(U_1) \geq 2,~~\dim(U_2) \geq 2,~~\mbox{and}~~\dim(U_1) + \dim(U_2) - \dim(U_3) \geq 5.\]
Then, there exist nonzero vectors $x_1 \in U_1$ and $x_2 \in U_2$ such that $\langle x_1, x_2 \rangle = 0$ and
\[\dim( U_3 \cap (x_1)^\perp \cap (x_2)^\perp ) \geq \dim(U_3)-1.\]
\end{lemma}

\begin{proof}
Let $U_1, U_2, U_3 \subseteq \Fset^t$ be three subspaces as in the statement of the lemma.

Assume first that $\dim(U_1 \cap U_2) \geq 3$.
In this case, there exist three linearly independent vectors in $U_1 \cap U_2$.
By Lemma~\ref{lemma:ortho_vec}, there exists a nonzero self-orthogonal linear combination of them.
By choosing $x_1$ and $x_2$ to be this vector, it obviously holds that $\langle x_1, x_2 \rangle = 0$ and that
\[\dim( U_3 \cap (x_1)^\perp \cap (x_2)^\perp ) = \dim( U_3 \cap (x_1)^\perp )\geq \dim(U_3)-1,\]
as required.

Assume next that $\dim(U_1 \cap U_3^\perp) \geq 1$. Here, $x_1$ can be chosen as an arbitrary nonzero vector of $U_1 \cap U_3^\perp$, and $x_2$ as an arbitrary nonzero vector of $U_2$ satisfying $\langle x_1, x_2 \rangle = 0$. Such a vector exists because $\dim(U_2) \geq 2$. By $x_1 \in U_3^\perp$, it follows that
\[\dim( U_3 \cap (x_1)^\perp \cap (x_2)^\perp ) = \dim( U_3 \cap (x_2)^\perp )\geq \dim(U_3)-1.\]
The case $\dim(U_2 \cap U_3^\perp) \geq 1$ is handled similarly.

Otherwise, we have $\dim(U_1 \cap U_2) \leq 2$ and $\dim(U_1 \cap U_3^\perp) = \dim(U_2 \cap U_3^\perp) =0$.
This implies that
\begin{eqnarray*}
\dim(U_1 + U_2) &=& \dim(U_1) + \dim(U_2) - \dim(U_1 \cap U_2)\\
&\geq& \dim(U_1) + \dim(U_2) - 2 \geq \dim(U_3)+3,
\end{eqnarray*}
where for the last inequality we have used the assumption $\dim(U_1) + \dim(U_2) - \dim(U_3) \geq 5$.
It thus follows that
\begin{eqnarray*}
\dim( (U_1 + U_2) \cap U_3^\perp) &=& \dim(U_1 + U_2) + \dim(U_3^\perp) - \dim(U_1 + U_2 + U_3^\perp)\\
&\geq& \dim(U_1 + U_2) + (t- \dim(U_3)) - t \geq 3.
\end{eqnarray*}
Hence, there exist vectors $w_1, w_2, w_3 \in U_1$ and $z_1, z_2, z_3 \in U_2$ for which the three sums
\[w_1 + z_1,~ w_2 + z_2,~ w_3 + z_3\]
are linearly independent vectors that belong to $U_3^\perp$.
By Lemma~\ref{lemma:ortho_vec}, there exist $\alpha_1, \alpha_2, \alpha_3 \in \Fset$, not all zeros, such that the vectors $x_1 = \sum_{i \in [3]}{\alpha_i \cdot w_i}$ and $x_2 = \sum_{i \in [3]}{\alpha_i \cdot z_i}$ satisfy $\langle x_1,x_2 \rangle =0$.
These vectors further satisfy that
\[x_1+x_2  = \sum_{i \in [3]}{\alpha_i \cdot (w_i+z_i)} \in U_3^\perp.\]
It follows that $x_1+x_2$ is nonzero, because the vectors $w_i+z_i$ are linearly independent.
Observe that $x_1$ belongs to $U_1$ and that it is nonzero, because otherwise the vector $x_2$ would be a nonzero vector that belongs to $U_2 \cap U_3^\perp$, in contradiction to $\dim(U_2 \cap U_3^\perp) =0$. By the same reasoning, $x_2$ is a nonzero vector of $U_2$.
Finally, notice that for every vector $u \in U_3$ such that $\langle u, x_1 \rangle =0$, it also holds that $\langle u, x_2 \rangle =0$, and thus
\[\dim( U_3 \cap (x_1)^\perp \cap (x_2)^\perp ) = \dim( U_3 \cap (x_1)^\perp ) \geq \dim(U_3)-1,\]
so we are done.
\end{proof}

\begin{remark}\label{remark:F_2}
It can be shown that for the binary field $\Fset_2$, the third condition of Lemma~\ref{lemma:choice_3sub} can be slightly weakened to $\dim(U_1) + \dim(U_2) - \dim(U_3) \geq 4$.
\end{remark}

We are ready to prove the following result, which implies Theorem~\ref{thm:n-sqrt{n}}.
\begin{theorem}\label{thm:K_n}
For an integer $k \geq 1$, put $n=k^2+2k+3$.
Then, for every finite field $\Fset$,
\[\chv(K_n,\Fset) \leq n - k.\]
\end{theorem}

\begin{proof}
For an integer $k \geq 1$ and a finite field $\Fset$, consider the complete graph $K_n$ on $n=k^2+2k+3$ vertices.
Let $V = A \cup B \cup C$ be the vertex set of the graph, where $A = \{v_1, \ldots, v_k\}$ is a set of $k$ vertices, $B$ is a set of $k^2+k$ vertices, and $C$ consists of the three remaining vertices.
To prove that $\chv(K_n,\Fset) \leq n - k$, suppose that for some integer $t$, we are given a subspace $U_v \subseteq \Fset^t$ with $\dim(U_v) = n-k$ for every $v \in V$. Our goal is to show that there exist pairwise orthogonal nonzero vectors $x_v \in U_v$ for $v \in V$.
We describe now a process with several steps for choosing the vectors. Throughout the process we maintain for every vertex $v \in V$ a subspace $U'_v$ defined as the subspace of the vectors currently available to the vertex $v$. Namely, for every partial choice of vectors, $U'_v$ is the subspace of $U_v$ that consists of all the vectors of $U_v$ that are orthogonal to all the previously chosen vectors. Initially, we have $U'_v = U_v$ for every $v \in V$.

Consider some partition of the set $B$ into $k$ sets, $B = B_1 \cup \cdots \cup B_k$, where $|B_i| = 2 \cdot (k-i+1)$ for every $i \in [k]$. Note that this is possible, because $|B| = k^2+k = 2 \cdot \sum_{i=1}^{k}({k-i+1)}$.
Our process starts with $k$ initial steps, where the role of the $i$th step ($i \in [k]$) is to choose vectors for the vertices of $B_i$ in a way that poses only $k-i+1$ linear constraints on the choice of the vector for $v_i$. Note that for the other vertices, the choice of the vectors for the vertices of $B_i$ might pose twice this number of linear constraints.

For $i \in [k]$, the $i$th step is performed as follows.
Consider an arbitrary partition of the set $B_i$ into $k-i+1$ pairs, denoted by $(a_1, b_1), \ldots, (a_{k-i+1},b_{k-i+1})$.
For every $j \in [k-i+1]$, we choose two nonzero vectors $u_{a_j} \in U'_{a_j}$ and $u_{b_j} \in U'_{b_j}$ such that $\langle u_{a_j}, u_{b_j} \rangle = 0$ and
\[\dim(U'_{v_i} \cap (u_{a_j})^\perp \cap (u_{b_j})^\perp ) \geq \dim(U'_{v_i})-1.\]
Observe that such a choice, if it exists, satisfies that $u_{a_j}$ and $u_{b_j}$ are nonzero vectors that belong to the subspaces of the vertices $a_j$ and $b_j$ respectively, they are orthogonal to all the previously chosen vectors and to one another, and in addition, their choice reduces the dimension of $U'_{v_i}$ by at most $1$.
To prove the existence of such a choice we apply Lemma~\ref{lemma:choice_3sub}.
The number of vectors chosen before the $(i,j)$ iteration is $\sum_{l=1}^{i-1}{|B_l|}+2 (j-1)$, hence each of $\dim(U'_{a_j})$ and $\dim(U'_{b_j})$ is at least \[(n-k)-\sum_{l=1}^{i-1}{|B_l|}-2(j-1).\]
Additionally, since the $2(j-1)$ already chosen vectors of the $i$th step reduce the dimension of $U'_{v_i}$ by at most $j-1$, it can be assumed that
\[\dim(U'_{v_i}) = (n-k)-\sum_{l=1}^{i-1}{|B_l|}-(j-1).\]
It thus follows that in the $(i,j)$ iteration, it holds that
\begin{eqnarray*}
\dim(U'_{a_j}) + \dim(U'_{b_j}) - \dim(U'_{v_i}) & \geq & (n-k)-\sum_{l=1}^{i-1}{|B_l|}-3(j-1) \\
& = & \sum_{l=i}^{k}{|B_l|} +3 -3(j-1)\\
& = & (k-i+1)(k-i+2) -3j +6\\
&\geq& (k-i+1)(k-i+2) -3(k-i+1)+6\\
& \geq & (k-i+1)(k-i-1) +6 = (k-i)^2 +5 \geq 5,
\end{eqnarray*}
where for the first equality we use the fact that $n = |B|+k+3$, and for the second inequality we use the fact $j \leq k-i+1$.
The above bound, which also implies that $\dim(U_{a_j}) \geq 2$ and that $\dim(U_{b_j}) \geq 2$, allows us to apply Lemma~\ref{lemma:choice_3sub} and to obtain the required vectors $u_{a_j}$ and $u_{b_j}$.

We next show that given the above choice for the vertices of $B$, one can choose vectors for the vertices of $A \cup C$ to obtain the required pairwise orthogonal vectors.
First, for the three vertices of $C$, choose arbitrary pairwise orthogonal nonzero vectors from the currently available subspaces. This is indeed possible, because so far we chose $n-(k+3)$ vectors, so the dimension of the subspace available to each of them is at least $3$. The choice for the first one leaves the available subspaces of the other two with dimension at least $2$, and the choice of the second one leaves the available subspace of the third with dimension at least $1$, allowing us to choose its nonzero vector.

Finally, we choose the vectors for the vertices of $A$.
For each $i \in [k]$, among the $n-k$ vectors chosen so far, there are $k-i+1$ pairs of vectors whose choice reduced the dimension of $U'_{v_i}$ by at most $1$.
This implies that we currently have
\[\dim(U'_{v_i}) \geq (n-k) -((n-k)-(k-i+1)) = k-i+1.\]
This allows us to go over the vertices $v_k, v_{k-1}, \ldots, v_1$, in this order, and to choose a nonzero vector from the subspace currently available to each of them, completing the proof.
\end{proof}

\begin{remark}
For the binary field $\Fset_2$, it can be shown that $\chv(K_n,\Fset_2) \leq n-k$ for $n=k^2+2k+2$.
This follows by applying the above proof with the version of Lemma~\ref{lemma:choice_3sub} mentioned in Remark~\ref{remark:F_2}.
\end{remark}

\section{Characterization of $2$-Subspace Choosable Graphs}\label{sec:2-vec}

In this section we prove Proposition~\ref{prop:chrachterization}, which asserts that for every finite field $\Fset$ and for every graph $G$, $\chv(G,\Fset) \leq 2$ if and only if $G$ contains no cycles.

\begin{proof}[ of Proposition~\ref{prop:chrachterization}]
If $G$ contains no cycles then it is $1$-degenerate, implying that it is $2$-subspace choosable over every field $\Fset$.
To complete the proof, we fix some finite field $\Fset$ and turn to show that for every $\ell \geq 3$, the $\ell$-cycle $C_\ell$ satisfies $\chv(C_\ell,\Fset)>2$.
We first prove it for $\ell=3$ and for $\ell=4$.
\begin{itemize}
  \item For $\ell=3$, assign to the vertices of the cycle $C_3$ the subspaces of $\Fset^3$ defined by
  \[U_1 = \linspan(e_1,e_2),~~~U_2 = \linspan(e_1,e_2+e_3), \mbox{~~~and~~~}U_3 = \linspan(e_1+\alpha \cdot e_3,e_2),\]
  where $\alpha \in \Fset$ is some field element to be determined.
  We claim that for some $\alpha \in \Fset$ it is impossible to choose three pairwise orthogonal nonzero vectors $x_i \in U_i$ ($i \in [3]$).
  Indeed, it is easy to verify that $x_1$ cannot be chosen as a scalar multiple of $e_1$ nor of $e_2$. So assume without loss of generality that $x_1$ is proportional to $e_1 + z \cdot e_2$ for some $z \neq 0$. If $x_1$ is orthogonal to $x_2$ and to $x_3$, then $x_2$ is proportional to $z \cdot e_1-e_2-e_3$ and $x_3$ is proportional to $z \cdot e_1+\alpha z \cdot e_3-e_2$. However, the inner product of the latter two is $z^2-\alpha \cdot z+1$, so it suffices to show that there exists $\alpha \in \Fset$ for which this quadratic polynomial has no root.
  Notice that in case that $z^2-\alpha \cdot z+1$ has a root, it can be written as $(z-\gamma)\cdot (z-\gamma^{-1})$ for some $\gamma \neq 0$. Since the number of possible values of $\alpha$ is larger than the number of possible invertible values of $\gamma$, it follows that the required $\alpha$ exists.

  \item For $\ell=4$, suppose first that the field $\Fset$ is of characteristic larger than $2$, and assign to the vertices along the cycle $C_4$ the subspaces of $\Fset^4$ defined by
  \[U_1 = U_2 = \linspan(e_1,e_2),~~~U_3 = \linspan(e_1+e_4,e_2+e_3), \mbox{~~~and~~~}U_4 = \linspan(e_1+\alpha \cdot e_3,e_2+e_4),\]
  where $\alpha \in \Fset$ is some nonzero field element to be determined.
  We claim that for some $\alpha \neq 0$ it is impossible to choose four nonzero vectors $x_i \in U_i$ ($i \in [4]$) that form a valid choice for $C_4$.
  By $\alpha \neq 0$, it is easy to verify, as before, that $x_1$ cannot be chosen as a scalar multiple of $e_1$ nor of $e_2$, so it can be assumed that it is proportional to $e_1 + z \cdot e_2$ for some $z \neq 0$. If the vectors $x_i$ form a valid choice for $C_4$, then $x_2$ is proportional to $z \cdot e_1-e_2$, thus $x_3$ is proportional to $e_1+e_4+z \cdot e_2+ z \cdot e_3$, and $x_4$ is proportional to $z \cdot e_1+\alpha z \cdot e_3-e_2-e_4$.
  However, the inner product of the latter two is $\alpha \cdot z^2-1$, so it suffices to show that there exists $\alpha \neq 0$ for which $\alpha \cdot z^2 \neq 1$ for all values of $z$. Since $\Fset$ is of characteristic larger than $2$, it has a non-square element, whose choice for $\alpha$ completes the argument.

  If, however, $\Fset$ is of characteristic $2$, one can consider the subspaces of $\Fset^5$ defined by
  $U_1 = U_2 =\linspan(e_1,e_2)$, $U_3 = \linspan(e_1+e_4,e_2+e_3+e_5)$, $U_4 = \linspan(e_1+ e_3,e_2+ \alpha \cdot e_4+e_5)$,
  where $\alpha \in \Fset$ is some nonzero field element for which $z^2+z \neq \alpha$ for all values of $z$. Notice that such an $\alpha$ exists because the function $z \mapsto z^2+z$ maps both $0$ and $1$ to $0$, so some nonzero element does not belong to its image.
  It can be verified that for the above subspace assignment, no valid choice of vectors for $C_4$ exists.
\end{itemize}
Finally, observe that for every odd $\ell > 3$, one can extend the above subspace assignment for $C_3$ by adding $\ell-3$ copies of the subspace $\linspan(e_1,e_2)$ between $U_1$ and $U_2$ to get a subspace assignment showing that $\chv(C_\ell,\Fset)>2$. Similarly, for every even $\ell > 4$, one can extend the above subspace assignment for $C_4$ by adding $\ell-4$ copies of the subspace $\linspan(e_1,e_2)$ between $U_1$ and $U_2$.
\end{proof}

\begin{remark}\label{remark:char_R}
As shown in~\cite{HaynesPSWM10}, the characterization given in Proposition~\ref{prop:chrachterization} for finite fields holds for the real field $\R$ too.
In particular, for every integer $\ell \geq 3$, it holds that $\chv(C_\ell,\R)>2$.
For an odd $\ell$, this simply follows by assigning $\R^2$ to every vertex.
For an even $\ell$, this follows from the construction given above in the proof for fields of characteristic larger than $2$, taking $\alpha$ to be some non-square over $\R$.
\end{remark}

\section{Subspace Choosability in Complete Bipartite Graphs}\label{sec:complete_bipartite}

In this section we prove our results on subspace choosability in complete bipartite graphs.

\subsection{Complete Balanced Bipartite Graphs}
Erd\"{o}s et al.~\cite{ERT79} proved that the choice number of the complete balanced bipartite graph $K_{m,m}$ exceeds $k$ for $m = \binom{2k-1}{k}$.
We provide here a quick proof for an analogue result for subspace choosability.
Note, however, that when the number of vertices is sufficiently large, the lower bound given by Theorem~\ref{thm:degreeIntro} is significantly better.
\begin{proposition}\label{prop:2k-1_k}
For every integer $k$ and for every field $\Fset$, $\chv(K_{m,m},\Fset) > k$ for $m = \binom{2k-1}{k}$.
\end{proposition}

\begin{proof}
Let $k$ be an integer and let $\Fset$ be a field.
Consider the graph $K_{m,m}$ for $m = \binom{2k-1}{k}$, and associate with the vertices of every side of the graph all the $k$-subsets of $[2k-1]$.
For a vertex associated with a $k$-subset $A$ of $[2k-1]$ we assign the $k$-subspace of $\Fset^{2k-1}$ spanned by the vectors $e_i$ with $i \in A$, where $e_i$ stands for the vector of $\Fset^{2k-1}$ with $1$ on the $i$th entry and $0$ everywhere else.
We claim that there is no choice of nonzero vectors from these subspaces such that the vectors of the left side are orthogonal to those of the right side. To see this, suppose in contradiction that such a choice exists, and denote by $x_1, \ldots, x_m$ and $y_1, \ldots, y_m$ the vectors chosen for the vertices of the left and right sides respectively. Letting $U = \linspan(x_1,\ldots,x_m)$ and $V = \linspan(y_1,\ldots,y_m)$, it follows that $V \subseteq U^\perp$, and thus
\[ \dim(U) + \dim(V) \leq \dim(U)+\dim(U^\perp) = 2k-1,\]
implying that at least one of $U$ and $V$ has dimension at most $k-1$.
Without loss of generality, assume that $\dim(U) \leq k-1$.
Put $\ell = \dim(U)$, fix some $\ell$ vectors from $x_1, \ldots, x_m$ that span $U$, and consider the $(2k-1) \times \ell$ matrix whose columns are these vectors.
Since the dimension of the subspace spanned by the rows of $U$ is also $\ell$, it follows that there exists a set $B \subseteq [2k-1]$ of $\ell$ indices whose rows are linearly independent.
It follows that the only vector in $U$ with zeros in all entries of $B$ is the zero vector.
However, by $|B| = \ell \leq k-1$, there exists a $k$-subset $A$ of $[2k-1]$ disjoint from $B$, so the vertex associated with this $A$ in the left side of the graph cannot receive any nonzero vector of $U$. This gives us the required contradiction and completes the proof.
\end{proof}

\subsection{Asymmetric Subspace Choosability in Complete Bipartite Graphs}

We consider now complete bipartite graphs in the asymmetric setting, where the dimensions of the subspaces assigned to the vertices of the right and left sides might be different.

\begin{definition}
The complete bipartite graph $K_{\ell_1,\ell_2}$ with the vertex set $A$ of size $\ell_1$ on the left side and the vertex set $B$ of size $\ell_2$ on the right side is said to be {\em $(k_1,k_2)$-subspace choosable} over a field $\Fset$ if it is $f$-subspace choosable over $\Fset$ for the function $f: A \cup B \rightarrow \{k_1,k_2\}$ defined by $f(u) = k_1$ for every $u \in A$ and $f(u) = k_2$ for every $u \in B$.
\end{definition}
\noindent
In what follows, we provide several conditions that imply subspace choosability and subspace non-choosability in complete bipartite graphs, and in particular prove Theorems~\ref{thm:Lower},~\ref{thm:GeneralUpper}, and~\ref{thm:Intro_K_2,m}.

\subsubsection{Upper Bounds}

We start with the following simple statement.
\begin{proposition}\label{prop:as_k_1_k_2}
For every field $\Fset$, the graph $K_{\ell_1,\ell_2}$ is $(k_1,k_2)$-subspace choosable over $\Fset$ whenever $\ell_1 < k_2$ or $\ell_2 < k_1$.
\end{proposition}

\begin{proof}
Suppose that $\ell_1 < k_2$, and let $U_1, \ldots, U_{\ell_1}$ and $V_1, \ldots, V_{\ell_2}$ be $k_1$-subspaces and $k_2$-subspaces, respectively, of $\Fset^t$ for some integer $t$. Choose an arbitrary nonzero vector from each $U_i$ for $i \in [\ell_1]$. Such a choice poses at most $\ell_1$ linear constraints on the choice of a vector from each $V_j$, and since the dimension of those subspaces is $k_2 > \ell_1$, a nonzero choice exists, resulting in a valid choice for the whole graph. By symmetry, the result holds for the case $\ell_2 < k_1$ as well.
\end{proof}

We next prove the following result, which confirms Theorem~\ref{thm:Lower}.

\begin{theorem}\label{thm:LowerAs}
For every two integers $k$ and $n$ and for every field $\Fset$, the graph $K_{k,m}$ is $(n,k)$-subspace choosable over $\Fset$ for $m = \sum_{i=0}^{k-1}{ \lfloor \frac{n-1}{k-i} \rfloor}$.
\end{theorem}

We need the following lemma.

\begin{lemma}\label{lemma:intersection}
Let $W$ be a $k$-subspace of some finite-dimensional vector space over a field $\Fset$, let $W_1, \ldots, W_t$ be $r$-subspaces of $W$, and suppose that $t \leq \frac{k-1}{k-r}$. Then, there exists a nonzero vector in the intersection $\bigcap_{i \in [t]}{W_i}$.
\end{lemma}

\begin{proof}
Using the standard equality $\dim(V_1 \cap V_2) = \dim(V_1)+\dim(V_2)-\dim(V_1+V_2)$, observe that
\begin{eqnarray*}
\dim \Big (\bigcap_{i \in [t]}{W_i} \Big ) &=& \dim(W_1)+ \dim \Big (\bigcap_{i \geq 2}{W_i} \Big ) - \dim \Big (W_1+\bigcap_{i \geq 2}{W_i} \Big )
\\ & \geq & \dim(W_1)+ \dim \Big (\bigcap_{i \geq 2}{W_i} \Big ) - \dim(W).
\end{eqnarray*}
By repeatedly applying this inequality we obtain that
\[ \dim \Big ( \bigcap_{i \in [t]}{W_i} \Big ) \geq \sum_{i \in [t]}{\dim(W_i)} - (t-1) \cdot \dim(W) = t \cdot r - (t-1) \cdot k \geq 1,\]
hence there exists a nonzero vector in $\bigcap_{i \in [t]}{W_i}$, as required.
\end{proof}

We are ready to prove Theorem~\ref{thm:LowerAs}.

\begin{proof}[ of Theorem~\ref{thm:LowerAs}]
For integers $k$ and $n$, put $m = \sum_{i=0}^{k-1}{ \lfloor \frac{n-1}{k-i} \rfloor}$.
We show that the graph $K_{k,m}$ is $(n,k)$-subspace choosable over every field $\Fset$.
Denote the left and right vertices of the graph by $u_1,\ldots,u_k$ and $v_1,\ldots, v_m$ respectively, and consider an arbitrary assignment of $n$-subspaces and $k$-subspaces of $\Fset^t$ to the left and right vertices, respectively, for some integer $t$.
For every $i \in [k]$ let $U_i$ be the subspace assigned to $u_i$, and for every $j \in [m]$ let $V_j$ be the subspace assigned to $v_j$.
We will show that it is possible to choose nonzero vectors from these subspaces such that the vectors of the left side are orthogonal over $\Fset$ to those of the right side.

We first describe how the vectors $x_1, \ldots, x_k$ of the left vertices $u_1, \ldots, u_k$ are chosen.
We choose them one by one, and to do so we maintain a set $J \subseteq [m]$ and some subspaces $L_1, \ldots, L_m$ of $\Fset^t$.
Initially, we define $J = [m]$ and $L_j = V_j^\perp$ for all $j \in [m]$. Note that $\dim(L_j)=t-k$.
Then, for every $i \in [k]$ we act as follows.
\begin{itemize}
  \item Pick some set $J' \subseteq J$ of size $|J'| = \lfloor \frac{n-1}{k-(i-1)} \rfloor$.
  \item Let $J'' \subseteq J'$ be the set of indices $j \in J'$ satisfying $\dim(L_j) = t-k+(i-1)$.
  \item Choose $x_i$ to be some nonzero vector of $U_i$ that belongs to the intersection $\bigcap_{j \in J''} L_j$.
  \item Add the vector $x_i$ to every subspace $L_j$, that is, update every subspace $L_j$ to be the subspace $L_j + \linspan(x_i)$.
  \item Remove the elements of $J'$ from $J$.
\end{itemize}
Observe that the number of elements removed from $J$ during the above $k$ iterations is $\sum_{i=1}^{k}{ \lfloor \frac{n-1}{k-(i-1)} \rfloor}$. Since the latter coincides with our definition of $m$, it follows that after the $k$th iteration the set $J$ is empty.

We show now that the vectors $x_i$ are well defined, in the sense that in the $i$th iteration there exists a nonzero vector that belongs to $U_i$ and to the intersection $\bigcap_{j \in J''} L_j$. To see this, put $W = U_i$ and consider its subspaces $W_j = L_j \cap W$ for $j \in J''$.
By the definition of $J''$, for every $j \in J''$ it holds that $\dim(L_j) = t-k+(i-1)$, hence, using $\dim(W)=n$, it follows that
\begin{eqnarray*}
\dim(W_j) &=& \dim(L_j)+\dim(W)-\dim(L_j+W) \\
&\geq& t-k+(i-1) + n-t=n-k+i-1.
\end{eqnarray*}
By Lemma~\ref{lemma:intersection} applied to $W$ and to its subspaces $W_j$, using the fact that $|J''| \leq \lfloor \frac{n-1}{k-(i-1)} \rfloor$, the required vector $x_i$ is guaranteed to exist.

We finally show that the above choice of vectors for the left vertices can be extended to a valid choice of vectors for the whole graph.
Fix some $j \in [m]$ and observe that if the subspace $L_j$ obtained at the end of the $k$th iteration has dimension strictly smaller than $t$ then it is possible to choose an appropriate vector $y_j$ for the vertex $v_j$. Indeed, $y_j$ can be chosen as any nonzero vector orthogonal to this $L_j$, because such a vector is orthogonal to $V_j^\perp$, hence belongs to $V_j$, and is orthogonal to all the vectors $x_1, \ldots, x_k$ that were chosen for the left vertices and were added to $L_j$ during the $k$ iterations. Since the initial dimension of $L_j$ is $t-k$ it suffices to show that in at least one of the $k$ iterations, the chosen vector $x_i$ was already inside $L_j$. So suppose that the set $J'$ includes $j$ in the $i$th iteration. If $j \in J''$ then the vector $x_i$ chosen in this iteration belongs to the current $L_j$. Otherwise, the dimension of $L_j$ in this iteration is smaller than $t-k+(i-1)$, implying that in one of the previous $i-1$ iterations a vector that already belongs to $L_j$ was chosen, so we are done.
\end{proof}

\subsubsection{Lower Bounds}

We start with the following simple statement.
\begin{proposition}\label{prop:basic_lower_bound}
For every two integers $n,k \geq 2$ and for every field $\Fset$, the graph $K_{k,n^k}$ is not $(n,k)$-subspace choosable over $\Fset$.
\end{proposition}

\begin{proof}
Denote the vertices of the left side of $K_{k,n^k}$ by $u_1,\ldots,u_k$.
For every $i \in [k]$, assign to the vertex $u_i$ the $n$-subspace of $\Fset^{n k}$ spanned by the vectors $e_i$ with $i \in [(i-1) \cdot n+1, i \cdot n]$, where $e_i$ stands for the vector in $\Fset^{n k}$ with $1$ on the $i$th entry and $0$ everywhere else.
Then, associate with each of the $n^k$ vertices of the right side a distinct $k$-tuple $(a_1,\ldots,a_k) \in [n]^k$, and assign to it the $k$-subspace of $\Fset^{n k}$ spanned by the vectors $e_{(i-1)n+a_i}$ for $i \in [k]$.

We claim that there is no choice of nonzero vectors from these subspaces such that the vectors of the left side are orthogonal over $\Fset$ to those of the right side.
To see this, consider any choice of a nonzero vector $x_i$ for each vertex $u_i$ for $i \in [k]$. For every $i \in [k]$, consider the restriction $\widetilde{x}_i \in \Fset^n$ of the vector $x_i$ to the support of its subspace, that is, to the entries with indices in $[(i-1) \cdot n+1, i \cdot n]$. Since $x_i$ is nonzero, it follows that there exists some $a_i \in [n]$ such that the vector $\widetilde{x}_i$ is nonzero in its $a_i$th entry. However, the only vector in the subspace of the vertex $(a_1, \ldots, a_k)$ of the right side which is orthogonal to all the vectors $x_i$ ($i \in [k]$) is the zero vector. This implies that no choice of nonzero vectors for the left side can be extended to a valid choice of vectors for the whole graph, so we are done.
\end{proof}

We next prove the following result.

\begin{theorem}\label{thm:GeneralUpper_as}
For every integers $n$, $t$, $k$ and for every field $\Fset$, the following holds.
If there exists a collection of $m = k \cdot (t-1) +1$ nonzero vectors in $\Fset^n$ satisfying that every $t$ of them span the entire space $\Fset^{n}$, then the graph $K_{k,m}$ is not $(n,k)$-subspace choosable over $\Fset$.
\end{theorem}

\begin{proof}
Suppose that there exists a collection of $m = k \cdot (t-1) +1$ nonzero vectors $b_1, \ldots, b_m$ in $\Fset^n$ satisfying that every $t$ of them span the space $\Fset^{n}$.
To prove that $K_{k,m}$ is not $(n,k)$-subspace choosable, we have to show that it is possible to assign $n$-subspaces and $k$-subspaces over $\Fset$ to the left and right vertices of the graph $K_{k,m}$ respectively, so that no choice of a nonzero vector from each subspace satisfies that the vectors of the left vertices are orthogonal over $\Fset$ to the vectors of the right vertices.

Let $u_1, \ldots, u_k$ be the vertices of the left side, and let $v_1, \ldots, v_m$ be the vertices of the right side.
For every $i \in [k]$, we assign to the vertex $u_i$ the subspace $U_i$ of $\Fset^{k n}$ that includes all the vectors whose support is contained in the entries indexed by $[(i-1) \cdot n+1,i \cdot n]$. In other words, viewing the vectors of $\Fset^{k n}$ as a concatenation of $k$ parts of length $n$, $U_i$ is the $n$-subspace of all the vectors that have zeros in all the parts but the $i$th one.
Then, for every $j \in [m]$, we assign to the vertex $v_j$ the subspace $V_j$ spanned by the $k$ vectors $e_1 \otimes b_j, \ldots, e_k \otimes b_j$ of $\Fset^{k n}$.
Here, $e_i$ stands for the vector in $\Fset^k$ with $1$ on the $i$th entry and $0$ everywhere else, and $\otimes$ stands for the tensor product operation of vectors.
Hence, $V_j$ is the $k$-subspace of all the vectors in $\Fset^{k n}$ consisting of $k$ parts, each of which is equal to the vector $b_j$ multiplied by some element of $\Fset$.

Assume for the sake of contradiction that there exist nonzero vectors $x_i \in U_i$ ($i \in [k]$) and $y_j \in V_j$ ($j \in [m]$) such that $\langle x_i,y_j \rangle = 0$ for all $i$ and $j$. For any $i \in [k]$, let $\tilde{x}_i \in \Fset^n$ be the (nonzero) restriction of the vector $x_i$ to the $i$th part.
For any $j \in [m]$, write $y_j = \sum_{i \in [k]}{\alpha_{i,j} \cdot e_i \otimes b_j}$ for some coefficients $\alpha_{i,j} \in \Fset$. Since all the vectors $y_j$ are nonzero, it clearly follows that at least $m$ of the coefficients $\alpha_{i,j}$ are nonzero.
Now, observe that for all $i \in [k]$ and $j \in [m]$, $\langle x_i,y_j \rangle = 0$ implies that $\langle \tilde{x}_i, \alpha_{i,j} \cdot b_j \rangle = 0$.
However, combining the facts that $\tilde{x}_i$ is nonzero and that every $t$ vectors among $b_1, \ldots, b_m$ span $\Fset^{n}$, it follows that for every $i \in [k]$, at most $t-1$ of the coefficients $\alpha_{i,j}$ with $j \in [m]$ are nonzero. This yields that the total number of nonzero coefficients $\alpha_{i,j}$ is at most $k \cdot (t-1) < m$, providing the desired contradiction.
\end{proof}

We derive the following.
\begin{corollary}\label{cor:Intro_m_as}
Let $k$ be an integer, and let $\Fset$ be a field.
\begin{enumerate}
  \item\label{itm_cor:1} For an integer $n$, set $m = k \cdot (n-1)+1$. If $|\Fset| \geq m$ then the graph $K_{k,m}$ is not $(n,k)$-subspace choosable over $\Fset$.
  \item\label{itm_cor:2} For integers $n$ and $q$, set $m = k \cdot \frac{q^{n-1}-1}{q-1}+1$.
  If $\Fset$ is a finite field of size $q \geq k$ then the graph $K_{k,m}$ is not $(n,k)$-subspace choosable over $\Fset$.
\end{enumerate}
\end{corollary}

\begin{proof}
For Item~\ref{itm_cor:1}, set $m = k \cdot (n-1)+1$, and let $\gamma_1, \ldots, \gamma_m$ be some distinct elements of the field $\Fset$.
For each $i \in [m]$, let $b_i$ be the vector in $\Fset^n$ defined by $b_i = (1, \gamma_i, \gamma_i^2, \ldots, \gamma_i^{n-1})$.
As follows from standard properties of the Vandermonde matrix, every $n$ of the vectors $b_1, \ldots, b_m$ are linearly independent and thus span the space $\Fset^n$.
By Theorem~\ref{thm:GeneralUpper_as} applied to these vectors with $t=n$, it follows that $K_{k,m}$ is not $(n,k)$-subspace choosable over $\Fset$, as required.

For Item~\ref{itm_cor:2}, set $t=\frac{q^{n-1}-1}{q-1}+1$ and $m = k \cdot (t-1)+1$.
Consider the equivalence relation on the nonzero vectors of $\Fset^n$ defined by calling two vectors equivalent if one is a multiple of the other by an element of $\Fset$.
Let $B$ be a collection of vectors in $\Fset^n$ that consists of one vector from every equivalence class, and note that $|B|=\frac{q^n-1}{q-1}$.
We observe that every $t$ vectors of $B$ span the space $\Fset^n$. Indeed, every strict subspace of $\Fset^n$ has dimension at most $n-1$, so it includes at most $q^{n-1}-1$ nonzero vectors, and thus at most $t-1$ vectors that represent different equivalence classes.
The assumption $q \geq k$ implies that
\[ m = k \cdot \frac{q^{n-1}-1}{q-1}+1 \leq \frac{q^n-1}{q-1} = |B|,\]
so by applying Theorem~\ref{thm:GeneralUpper_as} to $m$ of the vectors of $B$, we get that $K_{k,m}$ is not $(n,k)$-subspace choosable over $\Fset$, and we are done.
\end{proof}

Theorem~\ref{thm:GeneralUpper} and Corollary~\ref{cor:Intro_m} follow, respectively, from Theorem~\ref{thm:GeneralUpper_as} and Corollary~\ref{cor:Intro_m_as}.

We note that the approach proposed by Theorem~\ref{thm:GeneralUpper_as} for proving subspace non-choosability results seems to be more beneficial for large fields. This is justified by the following lemma that relates the size of the collection needed in Theorem~\ref{thm:GeneralUpper_as} to the size of the field.
Its proof is inspired by an argument given in~\cite{Ball12}.

\begin{lemma}
Let $\Fset$ be a finite field of size $q$, and let $m \geq t \geq n$ be integers.
If there exists a collection of $m$ nonzero vectors in $\Fset^n$ satisfying that every $t$ of them span the space $\Fset^{n}$, then
\[m \leq n-2+(q+1) \cdot (t-n+1).\]
\end{lemma}

\begin{proof}
Let $S \subseteq \Fset^n$ be a set of $m$ nonzero vectors in $\Fset^n$ satisfying that every $t$ of them span the space $\Fset^{n}$.
Let $x_1, \ldots, x_{n-2}$ be $n-2$ linearly independent vectors of $S$, and consider all the $(n-1)$-subspaces of $\Fset^n$ that include all of these vectors.
Observe that the number of such subspaces is $q+1$, and that these subspaces cover together the entire space $\Fset^n$. Since every $t$ vectors of $S$ span $\Fset^{n}$, it follows that each of these $q+1$ subspaces includes less than $t-(n-2)$ of the vectors of $S \setminus \{x_1,\ldots,x_{n-2}\}$. We thus conclude that $m = |S| \leq (n-2) + (q+1) \cdot (t-(n-2)-1)$, and we are done.
\end{proof}

\subsubsection{Two Vertices on the Left Side}

We next consider the particular case of the complete bipartite graph $K_{2,m}$ with two vertices on the left side and $m$ vertices on the right side.
It will be convenient to use the following definition.

\begin{definition}
The complete bipartite graph $K_{2,m}$ with the vertex set $A = \{u_1,u_2\}$ on the left side and the vertex set $B$ of size $m$ on the right side is said to be {\em $(k_1;k_2,k_3)$-subspace choosable} over a field $\Fset$ if it is $f$-subspace choosable over $\Fset$ for the function $f: A \cup B \rightarrow \{k_1,k_2,k_3\}$ defined by $f(u_1)=k_1$, $f(u_2)=k_2$ and $f(u) = k_3$ for every $u \in B$.
\end{definition}

In what follows, we prove Theorem~\ref{thm:Intro_K_2,m}.
We start with its first item, restated and proved below.

\begin{proposition}
For every integer $n$ and for every field $\Fset$, the graph $K_{2,n-1}$ is $(n;2,2)$-subspace choosable over $\Fset$.
\end{proposition}

\begin{proof}
For an integer $n$, consider the graph $K_{2,n-1}$.
To prove that it is $(n;2,2)$-subspace choosable over a field $\Fset$, consider for some integer $t$ arbitrary subspaces $U_1, U_2$ and $V_1, \ldots, V_{n-1}$ of $\Fset^t$ whose dimensions satisfy $\dim(U_1) = n$, $\dim(U_2) = 2$, and $\dim(V_j) = 2$ for $j \in [n-1]$.
Choose an arbitrary nonzero vector $x_2 \in U_2$, and for every $j \in [n-1]$ choose a nonzero vector $y_j \in V_j$ such that $\langle x_2, y_j \rangle = 0$. Note that this is possible since $\dim(V_j) = 2$. Finally, choose a vector $x_1 \in U_1$ satisfying $\langle x_1, y_j \rangle =0$ for all $j \in [n-1]$, which is possible by $\dim(U_1) = n$.
This gives us the required choice of vectors.
\end{proof}

By the above proposition, $K_{2,n-1}$ is $(n;2,2)$-subspace choosable over every field.
We consider the question of whether this holds even after adding another vertex to the right side of the graph.
Under certain conditions the answer is positive, as shown by the following result, confirming Item~\ref{itm:K_2,n_2} and the ``if'' part of Item~\ref{itm:K_2,n_3} in Theorem~\ref{thm:Intro_K_2,m}.

\begin{proposition}\label{prop:odd_n}
The graph $K_{2,n}$ is $(n;2,2)$-subspace choosable for every integer $n$ over $\C$ and for every odd integer $n$ over $\R$.
\end{proposition}

We need the following lemma, which is essentially given in~\cite{HaynesPSWM10}.

\begin{lemma}[{\cite[Lemma~2.9]{HaynesPSWM10}}]\label{lemma:Haynes}
Let $t \geq 2$ be an integer, and let $\Fset$ be either $\R$ or $\C$.
Let $U, V$ be two $2$-subspaces of $\Fset^t$ such that for every nonzero vector $x \in U$ there exists a nonzero vector $y \in V$ such that $\langle x,y \rangle \neq 0$.
Then, for every basis $u^{(1)},u^{(2)}$ of $U$ satisfying $\langle u^{(i)},u^{(j)} \rangle \neq 0$ if and only if $i =j$, there exists a basis $v^{(1)},v^{(2)}$ of $V$ satisfying $\langle v^{(i)},v^{(j)} \rangle \neq 0$ if and only if $i =j$ and, in addition, $\langle u^{(i)},v^{(j)} \rangle = 0$ if and only if $i =j$.
\end{lemma}

\begin{proof}[ of Proposition~\ref{prop:odd_n}]
Let $n$ be an integer, and let $\Fset$ be either $\R$ or $\C$.
Consider the graph $K_{2,n}$ with the vertex set $A = \{u_1,u_2\}$ on the left side and the vertex set $B = \{v_1,\ldots, v_n\}$ on the right side.
To prove that the graph is $(n;2,2)$-subspace choosable over $\Fset$, consider some subspaces $U_1, U_2, V_1, \ldots, V_n$ of $\Fset^t$ for some integer $t$, where $\dim(U_1)=2$, $\dim(U_2)=n$, and $\dim(V_j)=2$ for all $j \in [n]$. We will show now that there exist nonzero vectors $x_i \in U_i$ ($i \in [2]$) and $y_j \in V_j$ ($j \in [n]$) such that $\langle x_i , y_j \rangle =0$ over $\Fset$ for all $i$ and $j$.

Suppose first that there exists a nonzero vector $x_1 \in U_1$ such that $x_1$ is orthogonal to the subspace $V_{j'}$ for some $j' \in [n]$.
In this case, choose $x_1$ for the vertex $u_1$, and for every $j \in [n] \setminus \{j'\}$ let $y_j \in V_j$ be a nonzero choice for the vertex $v_j$ satisfying $\langle x_1, y_j \rangle =0$. Note that such a choice exists because $\dim(V_j)=2$.
These choices pose at most $n-1$ linear constraints on the choice for $u_2$, so by $\dim(U_2)=n$, there exists a nonzero vector $x_2 \in U_2$ that is orthogonal to all the vectors $y_j$ with $j \in [n] \setminus \{j'\}$. Finally, choose $y_{j'} \in V_{j'}$ as a nonzero vector orthogonal to $x_2$, whose existence is guaranteed by $\dim(V_{j'})=2$. The assumption on $x_1$ implies that $\langle x_1, y_{j'} \rangle =0$, so we obtain the required choice of vectors.

Otherwise, let $u_1^{(1)}, u_1^{(2)}$ be a basis of $U_1$ satisfying $\langle u_1^{(i)},u_1^{(j)} \rangle \neq 0$ if and only if $i =j$.
Since no nonzero vector of $U_1$ is orthogonal to some $V_j$, we can apply Lemma~\ref{lemma:Haynes} to obtain for every $j \in [n]$ a basis $v_j^{(1)}, v_j^{(2)}$ of $V_j$ that satisfies the assertion of the lemma. Note that it can be assumed that $\langle u_1^{(1)}, v_j^{(2)} \rangle = \langle u_1^{(2)}, v_j^{(1)} \rangle =1$ for all $j \in [n]$.
Let $M_1$ and $M_2$ be the $n \times t$ matrices over $\Fset$ whose $j$th rows are $v_j^{(1)}$ and $v_j^{(2)}$ respectively.

Now, to obtain the required choice of nonzero vectors, let $x_1 = \alpha \cdot u_1^{(1)} + \beta \cdot u_1^{(2)}$ be our nonzero choice for the vertex $u_1$ for some $\alpha,\beta \in \Fset$ to be determined.
Observe that this choice forces us to choose, up to a multiplicative constant, the vector $y_j = \alpha \cdot v_j^{(1)} - \beta \cdot v_j^{(2)}$ for the vertex $v_j$ for each $j \in [n]$.
For the vertex $u_2$, let $U \in \Fset^{t \times n}$ denote a matrix whose columns form a basis of the subspace $U_2$, and denote its choice by $x_2 = U \cdot \gamma$ for $\gamma \in \Fset^{n}$.
We consider the question of whether there exist $\alpha,\beta$ as above and a nonzero $\gamma$ such that $\langle x_2, y_j \rangle = 0$ for all $j \in [n]$.
Observe that this condition is equivalent to
\[(\alpha \cdot M_1 - \beta \cdot M_2 ) \cdot (U \cdot \gamma) = 0.\]
Letting $M'_1$ and $M'_2$ be the $n \times n$ matrices defined by $M'_1 = M_1 \cdot U$ and $M'_2 = M_2 \cdot U$, we ask whether there exist $\alpha, \beta \in \Fset$, that are not both zeros, and a nonzero vector $\gamma \in \Fset^n$ satisfying
\[(\alpha \cdot M'_1 - \beta \cdot M'_2 ) \cdot \gamma = 0.\]
If $\det(M'_1) = 0$ then we can take $\alpha = 1$ and $\beta = 0$, for which a nonzero $\gamma$ is guaranteed to exist.
Otherwise, if $\det(M'_1) \neq 0$, we take, say, $\beta = -1$, and show that for some $\alpha \in \Fset$, the matrix $\alpha \cdot M'_1 + M'_2$ is singular, implying the existence of the required vector $\gamma$. To see this, observe that $\alpha \cdot M'_1 + M'_2$ is singular if and only if $\alpha \cdot I_n + N$ is singular as well, where  $N = M'_2 \cdot (M'_1)^{-1}$. This reduces our question to whether for some $\alpha \in \Fset$ it holds that $\det(\alpha \cdot I_n + N)=0$ . This determinant is a degree $n$ polynomial in $\alpha$. Over $\Fset = \C$, this polynomial clearly has a root, and over $\Fset = \R$, assuming that $n$ is odd, it has a root as well. This completes the proof.
\end{proof}

We end this section by proving that adding a vertex to the right side of $K_{2,n-1}$ for an even integer $n$ results in a graph which is no longer $(n;2,2)$-subspace choosable over the real field $\R$ and over every finite field. This, in particular, gives us the ``only if'' part of Item~\ref{itm:K_2,n_3} of Theorem~\ref{thm:Intro_K_2,m}.

\begin{proposition}
Let $\Fset$ be either $\R$ or any finite field.
Then, for every even integer $n$, the graph $K_{2,n}$ is not $(n;2,2)$-subspace choosable over $\Fset$.
\end{proposition}

\begin{proof}
For a field $\Fset$ as above, Proposition~\ref{prop:chrachterization} and Remark~\ref{remark:char_R} imply that $\chv(K_{2,2},\Fset) > 2$.
Hence, for some integer $t$, there exist $2$-subspaces $L_1, L_2, R_1, R_2 \subseteq \Fset^t$ such that no choice of nonzero vectors $\widetilde{x}_i \in L_i$ and $\widetilde{y}_j \in R_j$ for $i,j \in [2]$ satisfies $\langle \widetilde{x}_i , \widetilde{y}_j \rangle = 0$ for all $i,j$.

For an even integer $n = 2k$, we define a subspace assignment to the vertices of $K_{2,n}$ that lies in the vector space $\Fset^{t \cdot k}$ as follows.
To the left vertices we assign the subspaces $U_1, U_2 \subseteq \Fset^{t \cdot k}$ defined by
\[U_1 = \linspan(e_1 \otimes L_1, \ldots, e_k \otimes L_1)~~\mbox{and}~~U_2 = \Big (\sum_{i=1}^{k}{e_i} \Big ) \otimes L_2,\]
and to the right vertices we assign the subspaces $V_1, \ldots, V_n \subseteq \Fset^{t \cdot k}$, defined by
\[V_{2j-1} = e_j \otimes R_1 ~~\mbox{and}~~V_{2j} = e_j \otimes R_2\]
for each $j \in [k]$.
Note that $\dim(U_1) = n$, $\dim(U_2)=2$, and $\dim(V_j)=2$ for all $j \in [n]$.
Intuitively, the assignment is designed so that the $t$-dimensional restriction of $U_1,U_2, V_{2j-1}, V_{2j}$ to the $j$th block is the assignment $L_1, L_2, R_1, R_2$.

To complete the proof, we show that there is no choice of nonzero vectors $x_i \in U_i$ and $y_j \in V_j$ for $i \in [2]$ and $j \in [n]$ that satisfies $\langle x_i , y_j \rangle = 0$ for all $i,j$.
So suppose for contradiction that such a choice exists, and let $j \in [k]$ be an integer for which the restriction of $x_1$ to the $j$th block is nonzero. Denote by $\widetilde{x}_1, \widetilde{x}_2, \widetilde{y}_1, \widetilde{y}_{2}$ the restrictions of the vectors $x_1, x_2, y_{2j-1}, y_{2j}$ to the $j$th block. Observe that these are nonzero vectors that satisfy $\widetilde{x}_i \in L_i$, $\widetilde{y}_j \in R_j$, and $\langle \widetilde{x}_i , \widetilde{y}_j \rangle = 0$ for all $i,j \in [2]$, in contradiction.
\end{proof}

\section{Hardness Result}\label{sec:hardness}

In this section we prove our hardness result, given in Theorem~\ref{thm:hardness}.
We start by presenting a gadget graph that will be used in the proof.

\subsection{Gadget Graph}

The main component of our hardness proof is the $\exists$-graph defined as follows.

\begin{definition}[$\exists$-graph]\label{def:exist_gadget}
For any integers $n_1,n_2$, define the $\exists$-graph $H=H_{n_1,n_2}$ and the function $f_H:V(H) \to \{2,3\}$ as follows.
The graph consists of a vertex labelled $\inver$ with degree 2,
whose two neighbors serve as the starting points of two subgraphs to which we will refer as the top and bottom branches.
Each branch is composed of a sequence of $4$-cycles connected by edges, as described in the figure below.
In each branch, the vertex of largest distance from $\inver$ in every $4$-cycle but the first has a neighbor labelled $\outver$ and another neighbor separating it from the next $4$-cycle (except for the last $4$-cycle).
The numbers of $\outver$ vertices in the top and bottom branches are $n_1$ and $n_2$ respectively.
The function $f_H$ is defined on the vertices of $H$ as indicated in the figure.
\vspace*{3mm}
	
~~~~~
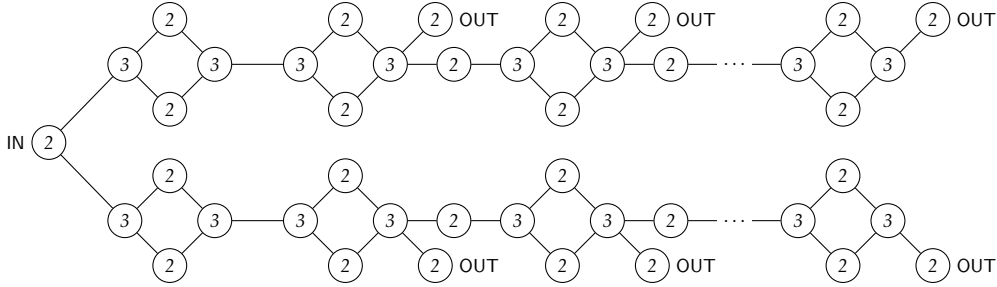
\begin{tikzpicture}[node distance={13mm}, main/.style = {draw, circle}, every node/.style={scale=0.65}]
\node[main] (IN) [label=left:{$\inver$}] {2};

\node[main] (UL0) [above of=IN, above=3mm, right=12mm] {3} edge (IN);
\node[main] (UB0) [below right of=UL0] {2} edge (UL0);
\node[main] (UT0) [above right of=UL0] {2} edge (UL0);
\node[main] (UR0) [above right of=UB0] {3} edge (UT0) edge (UB0);

\node[main] (UL1) [right of=UR0, right=1mm] {3} edge (UR0);
\node[main] (UB1) [below right of=UL1] {2} edge (UL1);
\node[main] (UT1) [above right of=UL1] {2} edge (UL1);
\node[main] (UR1) [above right of=UB1] {3} edge (UT1) edge (UB1);

\node[main] (UO1) [label=right:{$\outver$}, above right of=UR1] {2} edge (UR1);

\node[main] (UC1) [right of=UR1] {2} edge(UR1);

\node[main] (UL2) [right of=UC1] {3} edge (UC1);
\node[main] (UB2) [below right of=UL2] {2} edge (UL2);
\node[main] (UT2) [above right of=UL2] {2} edge (UL2);
\node[main] (UR2) [above right of=UB2] {3} edge (UT2) edge (UB2);

\node[main] (UO2) [label=right:{$\outver$}, above right of=UR2] {2} edge (UR2);

\node[main] (UC2) [right of=UR2] {2} edge(UR2);

\node[auto=false] (Udots) [right of=UC2] {$\ldots$} edge (UC2);

\node[main] (UL3) [right of=Udots] {3} edge (Udots);
\node[main] (UB3) [below right of=UL3] {2} edge (UL3);
\node[main] (UT3) [above right of=UL3] {2} edge (UL3);
\node[main] (UR3) [above right of=UB3] {3} edge (UT3) edge (UB3);

\node[main] (UO3) [label=right:{$\outver$}, above right of=UR3] {2} edge (UR3);

\node[main] (DL0) [below of=IN, below=3mm, right=12mm] {3} edge (IN);
\node[main] (DB0) [below right of=DL0] {2} edge (DL0);
\node[main] (DT0) [above right of=DL0] {2} edge (DL0);
\node[main] (DR0) [above right of=DB0] {3} edge (DT0) edge (DB0);

\node[main] (DL1) [right of=DR0, right=1mm] {3} edge (DR0);
\node[main] (DB1) [below right of=DL1] {2} edge (DL1);
\node[main] (DT1) [above right of=DL1] {2} edge (DL1);
\node[main] (DR1) [above right of=DB1] {3} edge (DT1) edge (DB1);

\node[main] (DO1) [label=right:{$\outver$}, below right of=DR1] {2} edge (DR1);

\node[main] (DC1) [right of=DR1] {2} edge(DR1);

\node[main] (DL2) [right of=DC1] {3} edge (DC1);
\node[main] (DB2) [below right of=DL2] {2} edge (DL2);
\node[main] (DT2) [above right of=DL2] {2} edge (DL2);
\node[main] (DR2) [above right of=DB2] {3} edge (DT2) edge (DB2);

\node[main] (DO2) [label=right:{$\outver$}, below right of=DR2] {2} edge (DR2);

\node[main] (DC2) [right of=DR2] {2} edge(DR2);

\node[auto=false] (Ddots) [right of=DC2] {$\ldots$} edge (DC2);

\node[main] (DL3) [right of=Ddots] {3} edge (Ddots);
\node[main] (DB3) [below right of=DL3] {2} edge (DL3);
\node[main] (DT3) [above right of=DL3] {2} edge (DL3);
\node[main] (DR3) [above right of=DB3] {3} edge (DT3) edge (DB3);

\node[main] (DO3) [label=right:{$\outver$}, below right of=DR3] {2} edge (DR3);
\end{tikzpicture}
\end{definition}

We need the following two claims.

\begin{claim}\label{claim:prop_3}
Let $\Fset$ be any field.
Let $A$ denote a neighbor of $\inver$ in the $\exists$-graph, and let $B$ denote another vertex adjacent to $A$.
Then, for every $f_H$-subspace assignment for $H$ over $\Fset$, there exists a choice of nonzero vectors for $\inver$ and $B$ which poses a single linear constraint on the choice for $A$.
\end{claim}

\begin{proof}
Let $W_{\inver}, W_A, W_B$ denote the subspaces assigned to the vertices $\inver,A,B$ respectively, and recall that $\dim(W_\inver)=2$, $\dim(W_A)=3$, and $\dim(W_B)=2$.
If there exists some nonzero vector in $W_{\inver} \cap W_B$, then choosing it for both $\inver$ and $B$ completes the proof.
Otherwise, it must hold that $\dim(W_{\inver}+W_B) = 4 > \dim(W_A)$, hence there exists some nonzero vector $x \in (W_{\inver}+W_B)\cap W_A^\perp$.
Write $x = x_1+x_2$ for $x_1\in W_{\inver}$ and $x_2\in W_B$.
If both of $x_1$ and $x_2$ are nonzero, choose them for $\inver$ and $B$.
Since every vector $y \in W_A$ satisfies $\inrprd{y,x} = 0$, it follows that if $\inrprd{y,x_1} = 0$ then $\inrprd{y,x_2} = 0$.
This implies that the only linear constraint that this choice poses on the vector of $A$ is the orthogonality to $x_1$.
If, however, $x_1$ is zero, then we have that $x_2 \in W_A^\perp$, so one can choose an arbitrary nonzero vector from $W_\inver$ for $\inver$ and $x_2$ for $B$.
Similarly, if $x_2$ is zero, we have that $x_1 \in W_A^\perp$, so one can choose $x_1$ for $\inver$ and an arbitrary nonzero vector from $W_B$ for $B$, completing the proof.
\end{proof}

\begin{claim}\label{claim:prop_4}
Let $\Fset$ be either $\R$ or any finite field, and let $x$ be either $e_6$ or $e_7$ in $\Fset^7$.
Then, there exists a subspace assignment $W_1,\ldots,W_4 \subseteq \Fset^7$ to the vertices $u_1, \ldots, u_4$ of $C_4$, with $\dim(W_1)= 3$ and $\dim(W_i)=2$ for $i \in \{2,3,4\}$, for which any valid choice of vectors assigns to $u_1$ a vector proportional to $x$.
\end{claim}

\begin{proof}
The proof of Proposition~\ref{prop:chrachterization} (see also Remark~\ref{remark:char_R}) describes for every field $\Fset$ as above, a $2$-subspace assignment for $C_4$ in $\Fset^5$ that admits no valid choice of vectors. Let $W_1,\ldots,W_4 \subseteq{\Fset^7}$ be the subspaces obtained from the subspaces of this assignment by adding two additional entries with values zero to their vectors.
Define $W'_1 = W_1+\linspan(x)$, and observe that any valid choice of vectors from the subspace assignment $W'_1,W_2, W_3, W_4$ assigns to $u_1$ a vector proportional to $x$, as otherwise, the restriction of such a choice to the first five entries would provide a valid choice for the given $2$-subspace assignment for $C_4$.
\end{proof}

The following lemma summarizes some properties of the $\exists$-graph.

\begin{lemma}\label{lemma:gadget_graph_properties}
	The $\exists$-graph $H$ and the function $f_H$ given in Definition~\ref{def:exist_gadget} satisfy the following.
	\begin{enumerate}
		\item\label{itm:1} The graph $H$ is bipartite, and every bipartition of $H$ puts all $\outver$ vertices in the same part.	
		\item\label{itm:2} For every $f_H$-subspace assignment for $H$ over any field $\Fset$, any choice of a nonzero vector for $\inver$ can be extended to all vertices of each of the branches.	
		\item\label{itm:3} For every $f_H$-subspace assignment for $H$ over any field $\Fset$ and for each of the branches of $H$,
		there exists a choice of a nonzero vector for $\inver$ which is compatible with any choice of vectors for the $\outver$ vertices of that branch.
		\item\label{itm:4} Let $\Fset$ be either $\R$ or any finite field, and let $t \geq 8$ and $j \in [t]$ be some integers. Then, there exists an $f_H$-subspace assignment for $H$ in $\Fset^t$ such that for every valid choice of vectors for $H$ there exists a branch all of whose $\outver$ vertices are assigned vectors proportional to $e_j$.
	\end{enumerate}
\end{lemma}

\begin{proof}
For Item~\ref{itm:1}, it can be easily seen that the graph defined in Definition~\ref{def:exist_gadget} is bipartite. Since the distance between every two $\outver$ vertices is even, it follows that every bipartition puts all of them in the same part.

For Item~\ref{itm:2}, consider some $f_H$-subspace assignment for $H$ over a field $\Fset$, and notice that any choice of a vector for $\inver$
reduces the dimension of the subspaces available to its neighbors by at most $1$.
So given any choice for $\inver$, one can choose, in each branch, an available nonzero vector for $\inver$'s neighbor, reducing the dimension of the subspaces available to its other neighbors to not less than $1$, allowing us to choose for them nonzero vectors as well. Their common neighbor has a subspace of dimension $3$, so the two chosen vectors of its neighbors reduce the dimension of the subspace available to it to not less than $1$, again allowing us to choose a nonzero vector.
Proceeding this way for vertices with increasing distances from $\inver$ allows us to choose vectors for all vertices of each of the branches of $H$.

For Item~\ref{itm:3}, consider some $f_H$-subspace assignment for $H$ over a field $\Fset$ and an arbitrary branch of $H$.
Let $A$ denote the neighbor of $\inver$ in this branch, let $B$ and $C$ denote the other neighbors of $A$, and let $D$ denote the remaining vertex of their $4$-cycle.
By Claim~\ref{claim:prop_3}, there exists a choice of nonzero vectors for $\inver$ and $B$ which poses a single linear constraint on the choice for $A$.
We claim that this choice for $\inver$ and $B$ is compatible with any choice of vectors for the $\outver$ vertices of that branch.
To see this, consider an arbitrary choice of nonzero vectors for these $\outver$ vertices.
The single neighbor of each $\outver$ vertex is assigned a $3$-subspace, so having made our choice for the $\outver$ vertices, each of these must still have a $2$-subspace from which its vector can be chosen.
Starting from the neighbor of the $\outver$ vertex of largest distance from $\inver$, we choose an arbitrary nonzero vector from its available $2$-subspace, allowing us to choose a nonzero vector for each of its two neighbors. Their other common neighbor has a $3$-subspace, so it includes a nonzero vector orthogonal to the vectors chosen for its neighbors. We proceed this way along the branch until we arrive to the $4$-cycle closest to $\inver$. Given the vectors chosen for the previous $4$-cycle and the choice for $B$, it is possible to choose from the $3$-subspace of $D$ some nonzero vector orthogonal to the vectors already chosen for its neighbors. Given this choice, we choose a nonzero vector orthogonal to it from the $2$-subspace of $C$, and since the choice for $\inver$ and $B$ poses a single linear constraint on $A$, it is possible to choose a nonzero vector for $A$ from its $3$-subspace. By Item~\ref{itm:2} of the lemma, our choice can be extended to the other branch, and we are done.

For Item~\ref{itm:4}, let $\Fset$ be either $\R$ or any finite field, and let $t \geq 8$ and $j \in [t]$ be some integers.
Assume without loss of generality that $j\ge 8$.
We define an $f_H$-subspace assignment for $H$ in $\Fset^t$ as follows.
The vertex $\inver$ is assigned the subspace $\linspan(e_6,e_7)$.
By Claim~\ref{claim:prop_4}, for $x$ being either $e_6$ or $e_7$ in $\Fset^7$, there exists a subspace assignment $W_1,\ldots,W_4 \subseteq \Fset^7$ to the vertices $u_1, \ldots, u_4$ of $C_4$, with $\dim(W_1)= 3$ and $\dim(W_i)=2$ for $i \in \{2,3,4\}$, for which any valid choice of vectors assigns to $u_1$ a vector proportional to $x$.
By extending these subspaces to $\Fset^t$ with zeros in the last $t-7$ entries, one can get such a subspace assignment in $\Fset^t$.
We put this subspace assignment with $x = e_6$ on the $4$-cycle closest to $\inver$ in each branch, where the $3$-subspace is assigned to the vertex with largest distance from $\inver$.
To the subspace of the top neighbor of $\inver$, we add the vector $e_6$, and to the one of the bottom, we add the vector $e_7$.
For all remaining $4$-cycles in the graph, we assign the subspaces of $\Fset^t$ given by Claim~\ref{claim:prop_4} with $x=e_7$, again with the $3$-subspace assigned to the vertex of largest distance from $\inver$, and add the vector $e_6$ to the subspace of the vertex closest to $\inver$.
Finally, to all $\outver$ vertices we assign the subspace $\linspan(e_7,e_j)$, and to the remaining vertices separating the 4-cycles, we assign the subspace $\linspan(e_6,e_7)$.

We claim that this $f_H$-subspace assignment for $H$ satisfies that for every valid choice of vectors there exists a branch all of whose $\outver$ vertices are assigned vectors proportional to $e_j$.
To see this, consider such a valid choice of vectors, and recall that it assigns to $\inver$ a nonzero vector from $\linspan(e_6,e_7)$.
In such a vector, at least one of the sixth and seventh entries is nonzero. We show that in the former case all the vectors of the $\outver$ vertices of the top branch are proportional to $e_j$. A similar argument shows that in the latter case, the same holds for the bottom branch.
Our assumption on the vector of $\inver$ implies that its neighbor in the top branch is orthogonal to $e_6$.
This essentially restricts its $4$-cycle to the subspace assignment given by Claim~\ref{claim:prop_4}, thus ensuring that the vertex of largest distance from $\inver$ in this $4$-cycle is assigned a vector proportional to $e_6$.
Applying this argument again to the next $4$-cycle yields that its vertex of largest distance from $\inver$ is assigned a vector proportional to $e_7$.
This ensures that the vector of its $\outver$ neighbor is proportional to $e_j$ and that the vector of its neighbor that separates its cycle from the next one is proportional to $e_6$. By repeating this argument for all the following $4$-cycles, the proof is completed.
\end{proof}

\subsection{Proof of Theorem~\ref{thm:hardness}}

To prove Theorem~\ref{thm:hardness}, we first prove the following.

\begin{theorem}\label{theorem:f_choosable_np_hard}
Let $\Fset$ be either $\R$ or any finite field.
It is $\NP$-hard to decide given a bipartite graph $G=(V,E)$ and a function $f: V \rightarrow \{2,3\}$ whether $G$ is $f$-subspace choosable over $\Fset$.
\end{theorem}

\begin{proof}
Let $\Fset$ be a field as in the statement of the theorem.
Given a $3$SAT formula $\phi$ with clauses $C_1,\ldots,C_m$ over the variables $x_1,\ldots, x_n$, we efficiently construct a graph $G_\phi=(V,E)$ and a function $f: V \to \{2,3\}$
such that $\phi$ is satisfiable if and only if $G_\phi$ is $f$-subspace choosable over $\Fset$.
Note that it can be assumed that each clause of $\phi$ contains three literals involving three distinct variables.
	
First, for each variable $x_j$, construct an $\exists$-graph $H_{n_1,n_2}$ (see Definition~\ref{def:exist_gadget}), where $n_1$ and $n_2$ are, respectively, the numbers of occurrences of the literals $x_j$ and $\overline{x_j}$ in $\phi$. Label the $\outver$ vertices of the top branch of $H_{n_1,n_2}$ by $x_j$, and the $\outver$ vertices of its bottom branch by $\overline{x_j}$. Define the function $f$ on the vertices of this graph as in Definition~\ref{def:exist_gadget}.
Next, for each clause $C_i$ of $\phi$, add a vertex representing $C_i$ and define its $f$ value to be $3$.
For each literal $x_j$ occurring in a clause $C_i$, add an edge between the vertex representing $C_i$ and a previously unchosen vertex labelled $x_j$, and likewise for the literals of the form $\overline{x_j}$.
Observe that $G_\phi$ is bipartite, as Item~\ref{itm:1} of Lemma~\ref{lemma:gadget_graph_properties} implies that there exists a bipartition placing all $\outver$ vertices of all $\exists$-graphs in the same part, thus the clause vertices may all belong to the opposite part.
Note that $G_\phi$ can be constructed in polynomial running time.

We prove now the correctness of the reduction.
Suppose first that there exists a satisfying assignment for $\phi$, and consider an arbitrary $f$-subspace assignment for $G_\phi$ over $\Fset$.
Then, for each variable $x_j$ with value $\true$, choose for the $\inver$ vertex of its $\exists$-graph a vector, promised by Item~\ref{itm:3} of Lemma~\ref{lemma:gadget_graph_properties}, which is compatible with any choice of vectors for the $\outver$ vertices labelled $x_j$.
If, however, $x_j$ has value $\false$, choose instead a vector for $\inver$ which is compatible with any choice of vectors for the $\outver$ vertices labelled $\overline{x_j}$.
By Item~\ref{itm:2} of the lemma, such a choice can be extended to all the vertices in the opposite branch.	
Now, since every clause has at most two literals which evaluate to $\false$ under the given satisfying assignment, we find that, so far, vectors have been chosen for at most two of the neighbors of each clause vertex. Since each clause vertex has a subspace of dimension $3$, we can make a choice for it which is compatible with all of its neighbors whose vectors have already been chosen. Observe that this choice can be extended to all the $\outver$ vertices for which no vectors have been chosen so far, because their subspaces have dimension $2$ whereas a vector has been chosen only for one of their neighbors. Finally, by our choice of the vectors of the $\inver$ vertices, using Item~\ref{itm:3} of Lemma~\ref{lemma:gadget_graph_properties}, one can properly choose vectors for the rest of the graph.
This implies that $G_\phi$ is $f$-subspace choosable over $\Fset$.

For the other direction, suppose that $G_\phi$ is $f$-subspace choosable over $\Fset$.
Put $t = n+7$, and apply Item~\ref{itm:4} of Lemma~\ref{lemma:gadget_graph_properties} to obtain an $f_H$-subspace assignment in $\Fset^t$ for each $\exists$-gadget, such that, for each $j \in [n]$, every valid choice of vectors assigns vectors proportional to $e_j$ either to all vertices labelled $x_j$ or to all vertices labelled $\overline{x_j}$.
Finally, to the vertex of a clause $C_i$ that involves the three variables $x_{j_1},x_{j_2},x_{j_3}$, assign the subspace spanned by $e_{j_1},e_{j_2},e_{j_3}$.
Since $G_\phi$ is $f$-subspace choosable over $\Fset$, there exists a valid choice for $G_\phi$ from these subspaces.
By our definition of the subspace assignment, for every $j \in [n]$, this choice assigns vectors proportional to $e_j$ to all vertices labelled $x_j$ or to all vertices labelled $\overline{x_j}$.
In the former case assign $x_j$ to $\false$, and in the latter to $\true$.
We claim that this assignment satisfies $\phi$.
To see this, observe that each vertex representing a clause $C_i$ must have for some $j \in [n]$ a neighbor labelled $x_j$ or $\overline{x_j}$ whose chosen vector is not proportional to $e_j$. This neighbor corresponds to a literal whose value is $\true$ according to our assignment, as desired.
\end{proof}

We also need the following simple lemma, whose proof employs ideas from~\cite{GutnerT09}.

\begin{lemma}\label{lemma:f_hardness_to_k_hardness}
For every field $\Fset$ and for every integer $k\ge 3$, the following holds. There exists a polynomial-time reduction from the problem of deciding for a given input of a bipartite graph $G=(V,E)$ and a function $f:V\to \{2,3\}$ whether $G$ is $f$-subspace choosable over $\Fset$, to the problem of deciding whether a given bipartite graph is $k$-subspace choosable over $\Fset$.
\end{lemma}

\begin{proof}
We start by proving the statement of the lemma for $k=3$.
Given a bipartite graph $G=(V,E)$ with bipartition $V = V_1 \cup V_2$ and given a function $f: V \rightarrow \{2,3\}$, consider the graph $G'$ that consists of nine copies of $G$, labelled $G_{i,j}$ for $i,j \in [3]$, and two additional vertices $v_1, v_2$ such that, for each $\ell \in \{1,2\}$, the vertex $v_\ell$ is adjacent to all vertices $u$ with $f(u)=2$ in the copies of $V_\ell$. It is easy to see that $G'$ is bipartite and that it can be constructed in polynomial running time.

For correctness, suppose first that $G$ is $f$-subspace choosable over $\Fset$, and consider an arbitrary assignment of $3$-subspaces over $\Fset$ to the vertices of $G'$.
Any choice of nonzero vectors for $v_1$ and $v_2$ will reduce the dimensions of the subspaces of the vertices of the graphs $G_{i,j}$ to not less than their original values under $f$.
Since each $G_{i,j}$ is $f$-subspace choosable over $\Fset$, it follows that there exists a valid choice of vectors for the vertices of $G'$, as required.
For the other direction, suppose that for some integer $t$, there exists an $f$-subspace assignment for $G$ such that no choice of nonzero vectors from the subspaces is valid. To the vertices of each subgraph $G_{i,j}$ in $G'$ we assign the subspaces of $\Fset^{t+3}$ obtained by adding three zeros to the head of all vectors of those subspaces.
To the subspaces of dimension $2$ in $G_{i,j}$, we add the vector $e_i$ for the vertices adjacent to $v_1$ and the vector $e_j$ for the vertices adjacent to $v_2$.
To each of the vertices $v_1$ and $v_2$ we assign the subspace of $\Fset^{t+3}$ spanned by $e_1,e_2,e_3$.
Now, for any choice of nonzero vectors for $v_1,v_2$, the subspaces of at least one of the graphs $G_{i,j}$ will be restricted to their initial $f$-subspace assignment, and will thus admit no valid choice of vectors for its vertices.

It remains to consider the case of $k>3$.
It suffices to show a polynomial-time reduction from the problem of deciding whether a given bipartite graph is $(k-1)$-subspace choosable over $\Fset$ to that of deciding whether a given bipartite graph is $k$-subspace choosable over $\Fset$. Here, given a bipartite graph $G=(V,E)$ with bipartition $V = V_1 \cup V_2$, consider the bipartite graph that consists of $k^2$ copies of $G$ and two additional vertices $v_1,v_2$ such that, for each $\ell \in \{1,2\}$, the vertex $v_\ell$ is adjacent to all the vertices in the copies of $V_\ell$. The correctness proof is similar to the one given above, so we omit the details.
\end{proof}

By combining Theorem~\ref{theorem:f_choosable_np_hard} with Lemma~\ref{lemma:f_hardness_to_k_hardness}, the proof of Theorem~\ref{thm:hardness} is completed.

\section*{Acknowledgements}
We thank the anonymous referees for their very helpful suggestions.
\bibliographystyle{abbrv}
\bibliography{choosability}

\end{document}